\documentclass{amsart}
\usepackage{amsfonts,amssymb,amsmath,amsthm, latexsym}
\usepackage{url}
\usepackage{enumerate}
\usepackage{graphicx}
\usepackage{color}

\newtheorem{theorem}{Theorem}[section]
\newtheorem{lemma}[theorem]{Lemma}
\newtheorem{proposition}[theorem]{Proposition}
\newtheorem{corollary}[theorem]{Corollary}
\theoremstyle{definition}
\newtheorem{definition}[theorem]{Definition}
\newtheorem{remark}[theorem]{Remark}
\numberwithin{equation}{section}

\author{Sergey Bezuglyi}
\address{Department of Mathematics, Institute for Low Temperature
Physics, Kharkiv 61103, Ukraine}
\curraddr{Department of Mathematics, University of Iowa, Iowa City,
52242 IA, USA}
\email{bezuglyi@gmail.com}

\author{Reem Yassawi}
\address{Department of Mathematics, Trent University, Peterborough, Canada}
\email{ryassawi@trentu.ca}


\usepackage{amsthm}
\graphicspath{{figures/}}

\theoremstyle{definition}
\newtheorem{example}[theorem]{Example}

\newcommand{\Z}{\mathbb Z}
\newcommand{\N}{\mathbb N}
\newcommand{\om}{\omega}
\newcommand{\wt}{\widetilde}
\newcommand{\ol}{\overline}
\newcommand{\e}{\varepsilon}

\theoremstyle{remark}

\begin{document}

\title{ Orders that yield homeomorphisms on Bratteli diagrams}

\begin{abstract}
 We call an order $\om$ on a Bratteli diagram $B$ {\em perfect} if its Vershik map is a homeomorphism. In this paper we  study the set of orders on a diagram and find necessary and sufficient conditions for an order to be perfect, in particular  when the order has several extremal paths. This work generalizes previous results, obtained for finite rank Bratteli diagrams.
  We describe an explicit procedure to
 create perfect orderings on Bratteli diagrams based on the study of
 certain relations between the entries of the diagram's incidence matrices and properties of the
 associated graphs, with the latter relations  characterizing diagrams which support perfect orderings. Also, we apply our theory to give a new combinatorial proof of the fact that the dimension group of a diagram supporting perfect orderings with $k$ maximal paths has a copy of  $\Z^{k-1}$ contained in their infinitesimal subgroup. Under certain conditions, we show that a similar result holds if the diagram supports countably many maximal paths.
Our results are illustrated  by numerous examples.
\end{abstract}
\maketitle

\section{Introduction}

  A {\em Bratteli diagram} $B$ (Definition \ref{Definition_Bratteli_Diagram}) is an infinite graph which encodes how a space $X_B$ is to be cut up into arbitrarily small pieces, in order that a dynamics be defined on it, by shifting these pieces. The space $X_B$ is represented as the set of infinite paths on $B$ starting at a distinguished vertex, and the dynamics is given by an {\em order} $\om$ (Section \ref{order}) which describes how those pieces  are to be shifted, or {\em stacked}. The dynamics $\varphi_\om$, called a {\em Vershik map,}  consists of moving up the stack, and
$\varphi_\om$ is defined and continuous everywhere except at points which are always at the top of a tower. In this article we are concerned with characterising {\em perfect} orders, which are those where  $\varphi_\om$ extends to a homeomorphism. We recall the following crucial fact that emphasizes the importance of perfect orderings:  the class of {\em aperiodic} (Definition \ref{aperiodic_definition}) Bratteli diagrams with perfect orderings is in a one-to-one correspondence with the set of aperiodic homeomorphisms of a Cantor set \cite{med}.

This paper is a natural extension of our previous work \cite{bky}
that was devoted to the study of perfect orderings on Bratteli diagrams of
{\em finite} rank.   We  summarize in Section \ref{finite_rank_section}  the relevant results from \cite{bky}, and refer to the article for a more detailed analysis.
We restrict our attention to {\em regular} Bratteli diagrams (Definition \ref{regular_definition}).  For these diagrams, if $\om$  is any order on $B$ that admits a Vershik map on  $X_B$, then this map is unique. For any order $\om$ on a  regular diagram $B$, the sets $X_{\max} (\om)$ and $X_{\min} (\om)$ are closed and nowhere dense in the path set $X_B$.

Suppose that  the Vershik map $\varphi_\om$ can be extended to a homeomorphism. Then  the order's  extremal edge structure, called its {\em skeleton} (Definition \ref{infinite_skeleton_definition}), constrains
 its {\em languages}  (Definition \ref{language_definition}): words in this language must generate paths in certain directed graphs $\mathcal H_n$ ({Definition \ref{infinite_associated_directed_graphs_definition})
that are defined by the skeleton. The skeleton also defines the extension of $\varphi_\om$ to the set of maximal paths of $\om$.  This extension must satisfy the properties described in Definition  \ref{correspondence_definition}; we call orders which satisfy these properties a {\em correspondence}.
  Conversely, given a Bratteli diagram,  skeleton and a correspondence, our main result is that we characterise,  in Theorem  \ref{analogous condition}, which orders extend to homeomorphisms, using the fact that  their language must correspond to valid Eulerian paths in $\mathcal H_n$.

The last section contains a new proof  of  known results stated in \cite{gps} which, in turn, are based on the earlier paper \cite{putnam}: namely that if $G$ is the dimension group of a simple Bratteli diagram  $B$, such that $B$ supports a perfect order with exactly $j$ maximal paths and $j$ minimal paths, then the infinitesimal subgroup of $G$ contains a subgroup isomorphic to $\mathbb Z^{j-1}$. Our combinatorial proof uses extensively the machinery of skeletons and associated graphs that we have developed, as well as our characterization of  diagrams that support perfect orderings  in Theorem \ref{analogous condition}.  We also show that our proof can be extended to a work for a class of Bratteli diagrams that have countably many extremal paths, and believe that  an appropriate version
of these results holds for dimension groups of aperiodic diagrams. Here we mention recent results of \cite{hand}, where given a dimension group whose infinitesimal subgroup contains $\Z^k$, concrete (equal row and column sum) Bratteli diagram representations of these dimension groups are found. Some of the examples in \cite{hand} can be shown to satisfy the conditions of Theorem \ref{analogous condition}. It would be interesting to characterize the dimension groups of diagrams that support perfect orderings.

We end with a few remarks. We find it useful to include a number of examples in the text that will help the reader to understand the concepts and statements.
Our examples are mainly of simple diagrams, but the constructs extend to aperiodic diagrams in a similar fashion.
The words  ``order'' and ``ordering'' are mostly used as synonyms, although we often use the former for a specific order, and the latter for an arbitrary order chosen from $\mathcal O_B$.

\section{Definitions and notation}\label{Preliminaries}

\subsection{ Bratteli diagrams}

\begin{definition}\label{Definition_Bratteli_Diagram}
A {\it Bratteli diagram} is an infinite graph $B=(V,E)$ such that the vertex
set $V=\bigcup_{i\geq 0}V_i$ and the edge set $E=\bigcup_{i\geq 1}E_i$
are partitioned into disjoint subsets $V_i$ and $E_i$ such that

(i) $V_0=\{v_0\}$ is a single point;

(ii) $V_i$ and $E_i$ are finite sets;

(iii) there exist a range map $r$ and a source map $s$ from $E$ to
$V$ such that $r(E_i)= V_i$, $s(E_i)= V_{i-1}$, and
$s^{-1}(v)\neq\emptyset$, $r^{-1}(v')\neq\emptyset$ for all $v\in V$
and $v'\in V\setminus V_0$. \end{definition}

The pair $(V_i,E_i)$ or just $V_i$ is called the $i$-th level of the
diagram $B$. A finite or infinite sequence of edges $(e_i : e_i\in E_i)$
such that $r(e_{i})=s(e_{i+1})$ is called a {\it finite} or {\it infinite
path} respectively. We write $e(v,v')$ to denote a finite path $e = (e_1,...,e_k)$ such
that $s(e) = v ( = s(e_1))$ and $r(e)=v' ( = r(e_k))$, and let $E(v,v')$ denote the set of all such paths. For a Bratteli diagram $B$,
let $X_B$ denote the set of infinite paths starting at the top vertex $v_0$. We
endow $X_B$ with the topology generated by cylinder sets
$U(e_k,\ldots,e_n):=\{x\in X_B : x_i=e_i,\;i=k,\ldots,n\}$, where$(e_k,\ldots,e_n)$ is a finite path in $B$ from level $k$ to level $n$. Then $X_B$ is a
0-dimensional compact metric space with respect to this topology. {\em We  assume throughout the paper  that $X_B$ has no isolated points for all considered Bratteli diagrams $B$.}

Given a Bratteli diagram $B$, the $n$-th {\em incidence matrix}
$F_{n}=(f^{(n)}_{v,w}),\ n\geq 0,$ is a $|V_{n+1}|\times |V_n|$
matrix whose entries $f^{(n)}_{v,w}$ are equal to the number of
edges between the vertices $v\in V_{n+1}$ and $w\in V_{n}$, i.e.
$$
 f^{(n)}_{v,w} = |\{e\in E_{n+1} : r(e) = v, s(e) = w\}|.
$$
Observe that every vertex $v\in V $ is
connected to $v_0$ by a finite path and the set $E(v_0,v)$ of all such
paths is finite. Set $h_v^{(n)}=|E(v_0,v)|$ for $v\in V_{n}$. Then
\begin{equation*}
h_v^{(n+1)}=\sum_{w\in V_{n}}f_{v,w}^{(n)}h^{(n)}_w \ \mbox{or} \ \
h^{(n+1)}=F_{n}h^{(n)}
\end{equation*}
where $h^{(n)}=(h_w^{(n)})_{w\in V_n}$.

Next we define some popular families of Bratteli diagrams.
We say that $B$ is {\em simple} if  for any level
$n$ there is $m>n$ such that $E(v,w) \neq \emptyset$ for all $v\in
V_n$ and $w\in V_m$.
 We say $B$ is  \textit{stationary} if $F_n = F_1$  for all $n\geq 2$.
We say $B$ has \textit{finite rank} if  for some $k$, $|V_n| \leq k$ for all $n\geq 1$.
Let $B$ have finite rank. We say $B$ has \textit{rank $d$} if  $d$ is the smallest
integer such that $|V_n|=d$ infinitely often.
We say that $B$ has {\em strict rank d} if $|V_n|=d$ for each $n\geq 1,$ and if $|r^{-1}(v)| \geq 2$ for each $v\in V\backslash V_0$.  If $V$ has strict rank $d$, we will assume that for any $n$ the vertex set $V_n$ does not depend on $n$ and equals the set $V$ that has exactly $d$ vertices.
In this article we will consider general Bratteli diagrams, which are not necessarily of finite rank.

For a Bratteli diagram $B$, the {\it tail (cofinal) equivalence}  relation
$\mathcal E$ on the path space $X_B$ is defined as $x\mathcal Ey$ if
 $x_n=y_n$ for all $n$ sufficiently large, where $x = (x_n)$, $y= (y_n)$.
Let $X_{per}=\{x\in X_B : |[x]_{\mathcal E}| <\infty\}$, with $[x]_{\mathcal E}$ denoting the $\mathcal E$ equivalence class of $x$.  By definition,
we have $X_{per}=\{x\in X_B : \exists n>0 \mbox{ such that }(|r^{-1}(r(x_i))|=
1\ \forall i\geq n)\}$.
A Bratteli diagram $B$ is called
\textit{aperiodic} if $X_{per}=\emptyset$, i.e., every
$\mathcal  E$-orbit is countably infinite.
\textit{Throughout the paper, we only consider aperiodic Bratteli diagrams $B$.  For these diagrams $X_B$ is a Cantor set and $\mathcal E$ is a Borel
equivalence relation on $X_B$ with uncountably infinitely many equivalence classes.}

Let $B$ be a Bratteli diagram, and $n_0 = 0 <n_1<n_2 < \ldots$ be a strictly increasing sequence of integers. The {\em telescoping of $B$ to $(n_k)$} is the Bratteli diagram $B'$, whose $k$-level vertex set $V_k'= V_{n_k}$ and whose incidence matrices $(F_k')$ are defined by
\[F_k'= F_{n_{k+1}-1} \circ \ldots \circ F_{n_k},\]
where $(F_n)$ are the incidence matrices for $B$.
Thus when one telescopes $B$ one takes a subsequence of levels
$\{n_k\}$ and considers the set $E(n_k, n_{k+1})$
 of all finite paths between the
 levels $\{n_k\}$ and $\{n_{k+1}\}$ as edges of the new diagram.
In particular, a Bratteli diagram $B$ has rank $d$ if and only if there is
a telescoping $B'$ of $B$ such that $B'$ has exactly $d$ vertices
at each level.

Next we define  a family of aperiodic diagrams that have an incidence matrix structure that is useful for our purposes.

\begin{definition}\label{aperiodic_definition}
We define the family $\mathcal A$ of Bratteli diagrams, all of whose incidence matrices are of the form
\begin{equation} \label{aperiodic_matrix_form}F_{n}:=
\left(
\begin{array}{ccccc}
A_{n}^{(1)} & 0 & \ldots  &  0 & 0 \\
0 & A_{n}^{(2)} & \ldots  & 0 & 0 \\
\vdots & \vdots  & \ddots & \vdots & \vdots     \\
0 & 0 & \ldots  & A_{n}^{(k)}  & 0 \\
B_{n}^{(1)}& B_{n}^{(2)}  & \ldots  & B_{n}^{(k)} & C_{n}  \\
 \end{array}
\right)
\end{equation}
where for some sequences  $(d_{1}^{(n)})_n, \ldots ,(d_{k}^{(n)})_n$ and  $(d^{(n)})_n,$
\begin{enumerate}
\item
$A_{n}^{(i)}$ is a $d_{i}^{(n+1)}\times d_{i}^{(n)}$ matrix,
\item all  matrices $A_{n}^{(i)}$, $B_{n}^{(i)}$ and $C_{n}$ are strictly
 positive, and
\item $C_{n}$ is a $d^{(n+1)}
\times d^{(n)}$ matrix.
\end{enumerate}

\end{definition}

\subsection{Orders on a Bratteli diagram}\label{order}

A Bratteli diagram $B=(V,E) $ is called {\it ordered}
if a linear order `$>$' is defined on every set  $r^{-1}(v)$, $v\in
\bigcup_{n\ge 1} V_n$. We use  $\om$ to denote the corresponding partial
order on $E$ and write $(B,\om)$ when we consider $B$ with the ordering $\om$. Denote by $\mathcal O_{B}$ the set of all orderings on $B$.

Every $\omega \in \mathcal O_{B}$ defines the  \textit{lexicographic}
ordering on the set $E(k,l)$
 of finite paths between
vertices of levels $V_k$ and $V_l$:  $(e_{k+1},...,e_l) > (f_{k+1},...,f_l)$
if and only if there is $i$ with $k+1\le i\le l$, $e_j=f_j$ for $i<j\le l$
and $e_i> f_i$.
It follows that, given $\om \in \mathcal O_{B}$, any two paths from $E(v_0, v)$
are comparable with respect to the lexicographic ordering generated by $\om$.  If two infinite paths are tail equivalent, and agree from the vertex $v$ onwards, then we can compare them by comparing their initial segments in $E(v_0,v)$. Thus $\om$ defines a partial order on $X_B$, where two infinite paths are comparable if and only if they are tail equivalent.
We call a finite or infinite path $e=(e_i)$ \textit{ maximal (minimal)} if every
$e_i$ is maximal (minimal) amongst the edges
from $r^{-1}(r(e_i))$.

Notice that, for $v\in V_i,\ i\ge 1$, the
minimal and maximal (finite) paths in $E(v_0,v)$ are unique. Denote
by $X_{\max}(\om)$ and $X_{\min}(\om)$ the sets of all maximal and
minimal infinite paths in $X_B$, respectively. It is not hard to show that
$X_{\max}(\om)$ and $X_{\min}(\om)$ are \textit{non-empty closed subsets} of
$X_B$; in general, $X_{\max}(\om)$ and $X_{\min}(\om)$ may have
interior points. For a finite rank Bratteli diagram $B$, the sets $X_{\max}(\om)$
and $X_{\min}(\om)$ are always finite for any $\om$, and if $B$ has rank $d$,
then each of them have at most $d$ elements
(\cite{bkms}).
An  ordered Bratteli diagram $(B, \om)$ is called \textit{properly ordered}
if the sets $X_{\max}(\om)$ and $X_{\min}(\om)$ are singletons.
We denote by $\mathcal O_B(j)$ the set of all orders on $B$ which have $j$ maximal and $j$ minimal paths.
Thus, in our notation, $\mathcal O_B(1)$ is the set of proper orders on $B$.

Let $(B,\omega)$ be an ordered Bratteli diagram, and suppose that $B'=(V',E')$ is the telescoping of $B$ to levels $(n_k)$. Let $v' \in V'$ and suppose that the two edges $e_1',$ $e_2'$, both with range $v'$, correspond to the finite paths $e_1$, $e_2$ in $B$, both with range $v$. Define the order $\om'$ on $B'$ by $e_1'<e_2'$ if and only if $e_1<e_2$. Then $\om' $ is called the {\em lexicographic order generated by $\om$} and is denoted by $\om'=L(\om)$.
It is not hard to see that if $\om' = L(\om)$, then
$$
|X_{\max}(\om)| = |X_{\max}(\om')|,\ \ |X_{\min}(\om)| = |X_{\min}(\om')|.
$$
If $B$ is an aperiodic Bratteli diagram and $\om \in \mathcal O_B$, then $X_{\max}(\om)\cap
X_{\min}(\om) =\emptyset$.

\begin{definition}\label{regular_definition}
A Bratteli diagram $B $ is called \textit{ regular} if for any ordering
$\omega \in \mathcal O_{B}$ the sets $X_{\max}(\omega)$ and
$X_{\min}(\omega)$ have empty interior.
\end{definition}
In particular, finite rank Bratteli
 diagrams are regular, and if all incidence matrix entries of $B$ are at least 2, then $B$ is regular.
{\em In this article we assume that all diagrams are regular.}

\begin{definition}\label{VershikMap}
Let $(B, \omega)$ be an ordered, regular Bratteli
diagram.  We say that $\varphi = \varphi_\omega : X_B\rightarrow X_B$
is a {\it  Vershik map} if it satisfies the following conditions:

(i) $\varphi$ is a homeomorphism of the Cantor set $X_B$;

(ii) $\varphi(X_{\max}(\omega))=X_{\min}(\omega)$;

(iii) if an infinite path $x=(x_1,x_2,\ldots)$ is not in $X_{\max}(\omega)$,
then $\varphi(x_1,x_2,\ldots)=(x_1^0,\ldots,x_{k-1}^0,\overline
{x_k},x_{k+1},x_{k+2},\ldots)$, where $k=\min\{n\geq 1 : x_n\mbox{
is not maximal}\}$, $\overline{x_k}$ is the successor of $x_k$ in
$r^{-1}(r(x_k))$, and $(x_1^0,\ldots,x_{k-1}^0)$ is the minimal path
in $E(v_0,s(\overline{x_k}))$.
\end{definition}

If $\om$ is an ordering on $B$, then one can always define the map
$\varphi_0$ that maps $X_B \setminus X_{\max}(\om)$ onto $X_B
\setminus X_{\min}(\om)$ according to (iii) of Definition
\ref{VershikMap}. The question about the existence of the Vershik map is
equivalent to that of an extension of  $\varphi_0 : X_B \setminus
X_{\max}(\om) \to X_B \setminus X_{\min}(\om)$ to a homeomorphism
of the entire set $X_B$.  If $\om$ is a proper ordering, then
$\varphi_\om$ is a homeomorphism. For a finite rank Bratteli diagram $B$, the
situation is simpler than for a general Bratteli diagram because the sets
$X_{\max}(\om)$ and $X_{\min}(\om)$ are finite. Note that Vershik did not assume that his maps where homeomorphisms, as he was working in the measurable context.
We say that an ordering $\om\in \mathcal O_{B}$ is \textit{perfect}
if $\om$ admits a Vershik map $\varphi_{\om}$ on $X_B$.
Denote by  $\mathcal P_B $ the set of
all perfect orderings  on $B$.
We observe that for a regular Bratteli diagram with an ordering $\om$, the Vershik map $\varphi_\om$, if it exists, is
defined in a unique way. Also, a necessary condition for $\om\in
\mathcal P_{B}$ is that $|X_{\max}(\om)|=|X_{\min}(\om)|$. Given $(B,\om)$ with $\om \in \mathcal P_B$, the uniquely defined  system $(X_B, \varphi_\om)$ is called  a {\em Bratteli-Vershik} system.

\subsection{The languages of an ordered Bratteli diagram}

If $V$ is a finite alphabet, let $V^{+}$ denote the set of nonempty
words over $V.$ We use the notation $W' \subseteq W$ to indicate that $W'$ is a subword of $W$.
 If $W_1, W_2, \ldots ,W_n$, are words, then we let  $\prod_{i=1}^{n}W_i$ refer to their concatenation.

Let $\om$ be an order on a Bratteli diagram $B$.  Fix a vertex $v\in
V_n$ and some level $m<n$, consider the set $E(V_m,v) = \bigcup_{v'\in V_{m}}E(v',v)$ of all finite paths between vertices of level $m$ and $v$. This set can be ordered by $\om$: $E(V_m, v) = \{e_{1},\ldots e_{p}\}$ where $e_{i}<e_{i+1}$ for
$1\leq i\leq p-1$. Define the word $w(v,m,n):=s(e_{1})s(e_{2})\ldots
s(e_{p})$ over the alphabet $V_{m}$. If $W=v_{1}\ldots v_{r} \in
V_{n}^{+}$, let $w(W,n-1, n):= \prod_{i=1}^{r}w(v_{i},n-1,n).$

\begin{definition}\label{language_definition}
The {\em level-$n$ language $\mathcal
L(B,\om, n)$ of $(B,\om)$ }  is
$$\mathcal L(B,\om,n):=
\{W :  \,\, W\subset w(v,n,N), \,\, \mbox{ for
some } v\in V_{N},  N>n\}\, .
$$

\end{definition}

If $B$ has strict rank $d$, then each of the level-$n$ languages can be defined on a common alphabet $V$, and in this case we
recover our definition of {\em the} language  $\mathcal L(B,\om)$ that we introduced in \cite{bky}:
 $$ \mathcal L(B, \om):=\limsup_{n}\mathcal L(B,\om, n)\, .
$$

\section{Skeletons on Bratteli diagrams}\label{skeletons}

\subsection{Skeletons and associated graphs on finite rank Bratteli diagrams}\label{finite_rank_section}
In this section we review definitions and results from \cite{bky} on finite rank diagrams. We do this mainly to set the stage for generalizing these notions to nonfinite rank diagrams, in Section \ref{infinite_rank_skeletons_etc}. Thus our discussion of finite rank ordered diagrams will be concise. For more details on definitions and results, we refer the reader to Section 3 in \cite{bky}.

Suppose that  $B$ has strict rank $d$.   If a maximal
(minimal) path $M$ ($m$) goes through the same vertex $v_{M}$
($v_{m}$) at each level of $B$, we will call this path {\em
vertical}. The following proposition characterizes when $\om$ is
a perfect order on a finite rank Bratteli diagram, and was proved in
(\cite[Proposition 3.2, Lemma 3.3]{bky}) for finite rank diagrams.

\begin{proposition}  \label{existence_vershik_map}
Let $(B, \om)$ be an ordered Bratteli diagram.

\begin{enumerate}
\item Suppose that $B$ has strict rank $d$ and that the
  $\om$-maximal  and $\om$-minimal paths  $M_1,...,M_k$ and
      $m_1,...,m_{k'}$
are vertical passing through the vertices $v_{M_1},\ldots , v_{M_k}$ and $v_{m_1},\ldots , v_{m_{k'}}$ respectively.  Then $\om$ is perfect if and only if

\begin{enumerate}\item $ k = k'$,

\item there is a permutation $\sigma$ of $\{1,\ldots k\}$ such that for
each $i\in \{1,...,k\}$, $v_{M_i}v_{m_{j}} \, \in \mathcal L(B,\om)$ if and only if $j=\sigma(i)$.

\end{enumerate}
\item
  Let  $B'$ be  a telescoping of
 $B$. Then $\omega \in \mathcal P_{B}$ if and only if $\om' = L(\om)\in \mathcal P_{B'}$.

\end{enumerate}
\end{proposition}

We give an example of how one would apply  Proposition \ref{existence_vershik_map}. This example also   answers negatively the following question, that is related to Statement 2 of Proposition \ref{existence_vershik_map}. {\em Let $B$ be a Bratteli diagram and $B'$ a telescoping of $B$. Is it true that any perfect order on $B'$ is obtained by telescoping of a perfect order on $B$?}

\begin{example}\label{perfect order by telescoping}
We define a stationary Bratteli diagram $B$ such that for a telescoped diagram $B'$ there is a perfect order $\om'\in \mathcal P_{B'}$ satisfying the condition  $\om' \neq L(\om)$ for any perfect order $\om$ on $B$.

Let $B$ be a stationary Bratteli diagram defined on the set of four vertices $\{a,b,c,d\}$ by the incidence matrix
$$F =\left(
  \begin{array}{cccc}
    2 & 1 & 1 & 1 \\
    1 & 2 & 1 & 1 \\
    1 & 1 & 2 & 1 \\
    1 & 1 & 1 & 2 \\
  \end{array}
\right).
$$

Let $B'$ be the diagram obtained by telescoping $B$ to all odd levels; it has  incidence matrix
 \[F' = F^2= \left(
  \begin{array}{cccc}
    7 & 6 & 6 & 6 \\
    6 & 7 & 6 & 6 \\
    6 & 6 & 7 & 6 \\
    6 & 6 & 6 & 7 \\
  \end{array}
\right).\]
 In order to define a perfect order $\omega'$, we let, for each $n$,  $w(a,n,n+1)=(adbc)^6a$, $w(b,n,n+1)=(bcad)^6b$,  $w(c,n,n+1)=(cadb)^6c$, and, finally, for $r^{-1}(d)$ we set  $w(d,n,n+1)= bcad^7 (bca)^5$. This appearance of $d^7$ prevents $\om'$ from being a  lexicographical order on $B'$ generated by any choice of $\om$ on $B$.  On the other hand, using Proposition \ref{existence_vershik_map},
we can verify that $\om' \in \mathcal P_{B'}$.  For, the set of words of length 2 that belong to $\mathcal L (B', \om')$ are $\{ab, ad, bc, ca, db, dd \}$
and $\sigma:\{a,b,c\} \rightarrow \{a,b,c\}$ defined by $\sigma(a)=b,$ $\sigma(b)=c$ and $\sigma ( c ) = a$ satisfies  part (1) of Proposition \ref{existence_vershik_map}.
\end{example}

Let  $\om$ be an order on a  Bratteli diagram $B$.     If  $v\in V\backslash V_0  $, we denote the  minimal edge with range $v$ by  $\ol e_v$ , and  we denote the maximal edge with range $v$ by $\wt e_v$.
\begin{definition} \label{well_telescoped_definition}
Let $(B,\om)$ be an ordered  rank $d$ diagram. We say that $(B,\om)$ is {\em well telescoped} if
\begin{enumerate}
\item $B$ has strict rank $d$,
\item all $\om$-extremal paths are vertical, with $\wt V$,  $\ol V$ denoting the sets of vertices through which maximal and minimal paths run respectively, and
\item $s(\wt e_v)\in \wt V$ and $s(\ol e_v) \in \ol V$ for each $v\in V\backslash (V_0 \cup V_1 )$, and this is independent of $n$.
\end{enumerate}
If $(B,\om)$ is perfectly ordered, for it to be  considered well telescoped, it will also have to satisfy
\begin{enumerate}
\setcounter{enumi}{3}
\item if $\wt v\ol v $ appears as a subword of some $w(v,m,n)$ with $m\geq 1$, then , then $\sigma(\wt v) = \ol v$ defines a one-to-one correspondence between the sets $\wt V$ and $\ol V$.
\end{enumerate}

\end{definition}
Given an ordered finite rank $(B,\om)$, it can always be telescoped so that it is well telescoped.
For details of how this can be done, see Lemma 3.11 in \cite{bky}. Thus, when we talk about a (finite rank) ordered diagram, we assume without loss of generality  that it is well telescoped.
For well telescoped ordered diagrams $(B,\om)$, we have $s(\wt e_v) \in \wt V_n$ and $s(\ol e_v) \in \ol V_n$ for any $v\in V_{n+1},\ n \geq 1$.
Given a well telescoped $(B,\om)$, we call the set $\mathcal F_{\omega} = (\wt V, \ol V, \{\wt e_{v}, \ol e_{v}: v\in V_{n}, \,\, n\geq 2\})$ the {\em skeleton
associated to $\om$}. If  $\om$ is a perfect order on $B$,   it follows that $|\ol V |= |\wt V|$, and
if $\sigma:\wt V\rightarrow \ol V$ is the
permutation given by Proposition \ref{existence_vershik_map}, we call
$\sigma$ the {\em accompanying permutation}.

The notion of a skeleton of an ordered diagram can be extended to an unordered diagram. Namely, given a strict rank $d$ diagram $B$, we select, two subsets $\wt V$ and $\ol V$ of $V$, of the same cardinality, and, for each $v\in V\backslash V_0 \cup V_1$, we select two edges $\wt e_v$ and $\ol e_v$, both with range $v$, and such that  $s(\wt e_v) \in \wt V$,
$s(\ol e_v) \in \ol V$. In this way we can extend the definition of a skeleton
$\mathcal F = (\wt V, \ol V, \{\wt e_{v}, \ol e_{v}: v\in V, \})$ to an unordered strict rank $d$ Bratteli diagram, with the objective of creating well-telescoped orders.
 A  more detailed discussion  can be found in \cite{bky}).  Arbitrarily choosing a
bijection $\sigma:\wt V \rightarrow \ol V$, we can consider the set of orders on $B$ which have $\mathcal F$ as skeleton and $\sigma$ as accompanying permutation.

 Given a skeleton $\mathcal F$ on a  finite rank diagram $B$,
for any vertices $\wt v \in
\wt V$ and $\ol v \in \ol V$, we set
\begin{equation}\label{W} W_{\wt v} = \{w \in V : s(\wt e_w) =
\wt v\} \mbox{ and }  W'_{\ol v} = \{w \in V : s(\ol e_w) =
\ol v\}.
\end{equation}
 Then $W = \{W_{\wt v} : \wt v \in \wt V\}$ and $W' = \{W_{\wt v}' : \ol v \in \ol V\}$ are both partitions of $V$.  We call $W$ and $W'$ the {\em partitions generated by $\mathcal F$.}
 Let $[\ol v, \wt v]:= W'_{\ol v} \cap W_{\wt v}$, and
define the partition \[W \cap W' := \{[\ol v,\wt v] : \ol v\in \ol V, \wt v \in \wt V, [\ol v,\wt v]\neq \emptyset\} .\]
Let $\mathcal F$  be a skeleton on
the strict finite rank  $B$ with
accompanying permutation $\sigma$.
Let  $\mathcal H=(T,P)$ be the directed graph where
the set $T$ of vertices of  $\mathcal H$  consists of partition elements $[\ol v, \wt v]$ of $W'\cap W$, and where there is an edge in  $P$ from $[\ol v, \wt v]$ to $[\ol v', \wt v']$
if and only if $\ol v' \, =\sigma(\wt v)$.
We call $\mathcal H$ the {\em  directed graph associated to $(B, \mathcal F, \sigma)$}.

\begin{example}\label{graph for two extremal paths} Suppose that $\wt V = \ol V = \{a,b\}$, $\sigma (a) = a$, $\sigma (b) = b$, and the set of vertices of $\mathcal H$ is $[a,a], [a,b], [b,a], [b,b]$.  Then $\mathcal H$ is illustrated in Figure \ref{first graph}.  Note that we do not specify a skeleton here. In general, it is  possible that for some skeletons one  of the vertices $[b,a]$ or $[a,b]$ be degenerate, for example, if $W_{a}'\cap W_{b}=\emptyset$, then the  vertex $[a,b]$ is not present in $\mathcal H$. If the permutation is $\sigma(a)=b$, $\sigma(b) = a$, then  $\mathcal H$ is identical to that of Figure \ref{first graph}, except that the vertices $\{[a,a], \,  [b,b], \,  [a,b], \,  [b,a]\}$ of the new graph are relabelled $\{[b,a],\, [a,b], \,[b,b],\, [a,a]\}$ respectively.
\end{example}

\begin{figure}[h]
\centerline{\includegraphics[scale=1.1]{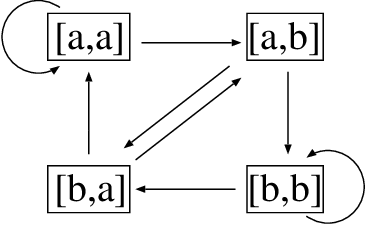}}
\caption{The associated graph $\mathcal H$ for $\om \in \mathcal P_{B}(2),$ $\sigma(a) = a$, $\sigma(b) = b$.
  \label{first graph}}
\end{figure}

Suppose that the strict finite rank $B$ has  skeleton  $\mathcal F$ and accompanying permutation $\sigma$. Then any path in $\mathcal H$ corresponds to a family of words in $V^+$: for  if $p_{1}
p_{2} \ldots p_{k}$ is a path in $\mathcal H$ where $p_{i}$ has source $[\ol v_i, \wt v_i]$, and for each $i$, $v_i$ is any vertex in $  [\ol v_i, \wt v_i]$,  then the path $p_{1}
p_{2} \ldots p_{k}$ corresponds to the word $v_1 v_2 \ldots v_k$ (and many such words can exist).
Conversely, if the word $v_1 v_2 \ldots v_k$ is such that $v_i \in [\ol v_i, \wt v_i]$ for each $i$, and
$[\ol v_1, \wt v_1]\, [\ol v_2, \wt v_2] \ldots [\ol v_k, \wt v_k]$ is a path in $\mathcal H$, then we say that $v_1v_2 \ldots v_k$ {\em corresponds to a valid path in $\mathcal H$}.
The relevance of $\mathcal H$ for perfect orders is described by the following lemma,
which  was proved  in \cite[Remark 3.16, Lemma 3.17]{bky}.

\begin{lemma}\label{graphs_determine_perfection}
 Let B be a
 strict  finite rank Bratteli diagram, $\mathcal F$ be a skeleton on $B$ and $\sigma : \wt V\rightarrow  \ol V$ be a bijection. Let $\mathcal H$ be the associated directed graph. Suppose that the ordering $\om$ on $B$ has skeleton and accompanying permutation $(\mathcal F, \sigma)$, and  such  that each  word in $\mathcal L(B, \om)$
corresponds to a valid path in $\mathcal H$. Then $\om$ is perfect. Conversely, if $(B,\om)$ is a well telescoped, perfectly ordered diagram with skeleton and accompanying permutation $(\mathcal F, \sigma)$, then every word in $\mathcal L(B, \om)$ corresponds to a valid path in $\mathcal H$.
\end{lemma}

\subsection{Skeletons, associated graphs, and correspondences on general Bratteli diagrams}\label{infinite_rank_skeletons_etc}
If $B $ is not of finite rank, the notion of a skeleton can be
generalized, although the notation is more technical.

\begin{definition}
Suppose that  $(B, \om)$ is an ordered Bratteli diagram.
To each maximal path $M$  and minimal path $m$ we
associate the sequences $(v_{n}(M))$ and $(v_{n}(m))$ of vertices that $M$  and $m$ pass
through. For each $n$, let
$\wt V_{n}:=\{v\in V_{n}: v = v_{n}(M) \mbox{ for some maximal path } M\}
$; we call vertices in $\wt V_{n}$ {\em maximal vertices}. Similarly we can define $\ol V_{n}$, the set of {\em minimal vertices} in $V_{n}$.
\end{definition}

 In other words,
for each $\wt v \in \wt V_{n}$,   there is at least  one  infinite maximal  path passing
through $\wt v$, and for each $\ol v \in \ol V_n$, there is at least one infinite minimal path passing through $\ol v$.
The following lemma tells us that the notion of a well telescoped ordered Bratteli diagram can be extended to general Bratteli diagrams:

\begin{proposition}\label{infinite_rank_skeleton}
Let $(B, \om)$ be an ordered aperiodic Bratteli diagram. Then there exists
a telescoping $(B', \omega')=((V',E'), \omega')$ of $(B,\omega)$ to a sequence of levels $(n_k)$,  such that
\begin{itemize}
\item
for every vertex $v\in V'$,  any
maximal edge $\wt e_{v} \in E_{k}'$ has  source
in  $\wt V_{k-1}'$ and any minimal edge $\ol e_v$ has source in $\ol V_{k-1}'$, and
\item all of $B$'s incidence matrix entries  that are nonzero are at least two.
\end{itemize}
\end{proposition}

\begin{proof} We use the idea of the proof of Proposition 2.8 from \cite{hps}.
Take $n_1=1$, write $V_{1} = \wt V_1 \cup (V_{1}\backslash \wt
V_1)$. Take a vertex $v \in V_{1}\backslash \wt V_1$ and consider all
maximal edges $e$ such that $s(e) = v$. We say that  a maximal edge $e$ is
{\em extendable} if there is a maximal  edge $e'$ such that $r(e) =
s(e')$. Let $E_v$ be the set of all finite maximal paths consisting of extendable edges starting from $v$. The set $E_v$ is  finite, and this is true for any $v \in V_{1}\backslash \wt V_1$.
 Therefore we can  find $n_2$ such that if $f$ is a maximal
  finite path with source in $V_{1}\backslash \wt V_1$, then $r(f)\in
  V_{n}$ with $n<n_{2}$. Thus, all maximal paths with source in $V_{1}$
  and range in $V_{n_2}$ must have source in $ {\wt V}_{1}$. We
  describe only the next step, the rest following by induction. Write
  $V_{n_2} = \wt V_{n_2} \cup (V_{n_2}\backslash \wt V_{n_2})$. Find
  an $n_3$ such that if $f$ is a maximal finite path with source in
  $V_{n_2}\backslash \wt V_{n_2}$, then $r(f)\in V_{n}$ with
  $n<n_{3}$. Therefore all maximal paths with source in $V_{n_2}$ and range
  in $V_{n_3}$ must have source in ${\wt V_{n_2}}$. Continue, and
  telescope $(B,\om)$ via levels $(n_{k})$. Since maximal edges in
  $E_{k}'$ in the telescoped diagram $(B', \om')$ correspond to
  maximal paths between $V_{n_{k-1}}$ and $V_{n_{k}}$ in $B$, the
  result follows. An identical argument yields the result for
  minimal edges.
\end{proof}

\begin{definition}
Let $(B, \om)$ be an ordered Bratteli diagram, where the sequences $(\wt V_n)$ and $(\ol V_n)$ consist of the maximal and minimal vertices, respectively. Suppose that for each  $n\geq 2 $ and each $v \in V_{n+1}$, $s(\wt e_v ) \in \wt V_n$ and $s(\ol e_v ) \in \ol V_n$. Suppose also that all of $B$'s nonzero incidence matrix  entries are at least two. Then we say that  $(B, \om)$ is {\em well telescoped}.
\end{definition}

 Proposition \ref{infinite_rank_skeleton} tells us that
we can assume that $(B,\om)$ is well telescoped.
Next we generalize the notion of a skeleton to an unordered, non-finite rank Bratteli diagram.
Let $B$ be a Bratteli diagram, where we assume that all nonzero entries of its incidence matrices are at least two.

\begin{definition}\label{infinite_skeleton_definition}
Let $\{M_\alpha: \alpha \in I\}$, and $\{m_\beta : \beta \in J\}, |I| = |J|$, be closed, nowhere dense sets of  infinite paths.
For $n \in \mathbb N$, let\[\wt V_n := \{\wt v :    \wt v = v_n(M_\alpha) \mbox{ for some }  \alpha \in  I  \} \mbox{ and  } \ol V_n := \{\ol v : \ol v = v_n(m_\beta) \mbox{ for some }  \beta \in J\}.\]
 Let
$\{\wt e_{v} : v \in V_{n+1} \}$ and
$\{\ol e_{v} : v \in V_{n+1}    \}$
be sets of edges such that
$r(\wt e_v) =r (\ol e_v) = v$,
$s(\wt e_v ) \in \wt V_n $ and $s(\ol e_v ) \in \ol V_n$.
Suppose that for any $n$ and any $N> n$,
\begin{enumerate}
\item
if $v_N(M_\alpha)= v_N(M_\alpha'),$ then $v_n (M_\alpha)= v_n(M_\alpha')$, with the analogous condition holding  for paths $m_\beta$ and $m_{\beta'}$,
\item if $v\in \wt V_{n}\cap \ol V_{n}$, then $\wt e_v \neq \ol e_v$, and \item  if $n \in \mathbb N$, $\alpha \in I$ and $\beta \in J$, then $M_\alpha \in U(\wt e_v)$ whenever $v= v_n(  M_\alpha  )$ and $m_{\beta} \in    U(\ol e_v ) $ whenever $v= v_n (m_\beta)$. 
\end{enumerate}
Then we call $\mathcal F = (\wt V_{n-1}, \ol V_{n-1}, \{\wt e_{v}, \ol e_{v}: v\in V_{n}\}: n \geq 2)$ the {\em skeleton associated to  $\{M_\alpha: \alpha \in I\}$ and $\{m_\beta : \beta \in J\}$.}
Vertices in the sets $\wt V_n$ and $\ol V_n$ are called {\em maximal} and {\em minimal} vertices respectively.
Paths $M_\alpha, \;\alpha \in I,$ are called {\em maximal} and form the set $X_{\max}(\mathcal F)$, and paths $m_\beta, \;\beta \in J,$ are called {\em minimal} and form the set $X_{\min}(\mathcal F)$.
\end{definition}

\begin{remark} We note the following:
\begin{itemize}
\item
If we are given a well telescoped ordered diagram $(B,\om)$ and $|X_{\max} (\om)|=| X_{\min} (\om)|$, then we can define the skeleton generated by $\om$, by letting
 $\{M_\alpha: \alpha \in I\} = X_{\max} (\om)$, $\{m_\beta : \beta \in J\}=X_{\min} (\om),$
 and choosing for each $v \in V\backslash V_0\cup V_1$,   $\wt e_v, \ol e_v $ to be the maximal and minimal edges with range $v$ respectively.
 It is clear that if $\mathcal F$ is defined by such an  order $\om$, then $X_{\max} (\mathcal F) = X_{\max} (\om)$ and $X_{\min} (\mathcal F) = X_{\min} (\om)$.
 \item
 We will not be  concerned with orders $\om$ such that $|X_{\max} (\om)|\neq| X_{\min} (\om)|$, as our aim is to characterize perfect orders and such a condition would prevent an order from being perfect.
 \item
 The requirement that $X_{\max} (\mathcal F)$ and $X_{\min} (\mathcal F)$ be closed is  natural: after all, we are introducing skeletons to build  orders, and  in that case the latter sets must be closed.
 \item
 The regularity of $B$ leads to the requirement that the sets $X_{\max}(\om)$ and $X_{\min}(\om)$   are nowhere dense.
\end{itemize}
\end{remark}

\begin{example}\label{basic_infinite_rank_example}
Suppose that $V_0=\{v_0 \}$, and for $n\geq 1$, $V_{n}=\wt{V}_n= \ol {V}_n= \{v_{1}, \ldots v_{n}\}$, and all nonzero incidence matrix entries are at least two.
Suppose also that  for $v_i \in V_{n+1}$, we choose $\wt e_i \neq \ol e_i$ and define
\[ s(\wt e_{v_{i}})
   = \left\{
\begin{array}{rl}
v_i
& \mbox{if $i \neq n+1$
   } \\
v_n & \mbox{ if $i=n+1$}
  \end{array}
  \right.
, \mbox{ and } \,\,\,s(\ol e_{v_{i}})
   = \left\{
\begin{array}{rl}
v_i
& \mbox{if $i \neq n+1$
   } \\
v_1& \mbox{ if $i=n+1$}
  \end{array}
  \right. \, . \]
 In Figure \ref{skeleton_example} we have drawn (only) the extremal edges in $B$,     with dashed lines represented maximal edges and solid lines representing minimal edges.
Consider the sets of infinite paths   $ \{  M_\alpha: \alpha \in \mathbb N \cup\{\infty \}\}  $ whose edges consist of the identified maximal edges, so that
$M_1$ passes  vertically through the vertex $v_1$ at all levels, and for $i>1$,
$M_i$ passes through vertices $v_1$, $v_2, \ldots v_{i-1}, v_i$,  and then goes down vertically through $v_i$. Finally, $M_\infty$ passes through vertices $v_1, v_2, v_3, \ldots$. Similarly consider the set of infinite paths $\{m_\beta : \beta \in \mathbb N\}$ whose edges consist of the identified minimal edges, so that
$m_1$ passes vertically through $v_1$, and for $i>1$,
$m_i$ passes through $v_1$ exactly $i-1$ times, then jumps to $v_i$ and goes down vertically through $v_i$. It is straightforward to verify that the sets  $ \{  M_\alpha: \alpha \in \mathbb N \cup\{\infty \}     \}  $ and $\{m_\beta : \beta \in \mathbb N\}$
 are both countable, closed, and nowhere dense, and that $\mathcal F = (\wt V_{n-1}, \ol V_{n-1}, \{\wt e_{v}, \ol e_{v}: v\in V_{n}\}: n \geq 2)$ is the  skeleton associated to  $\{M_\alpha: \alpha \in \mathbb N \cup \{ \infty \}\}$ and $\{m_\beta : \beta \in \mathbb N\}$.
  \end{example}

  \begin{figure}

 \centerline{\includegraphics[scale=0.7]{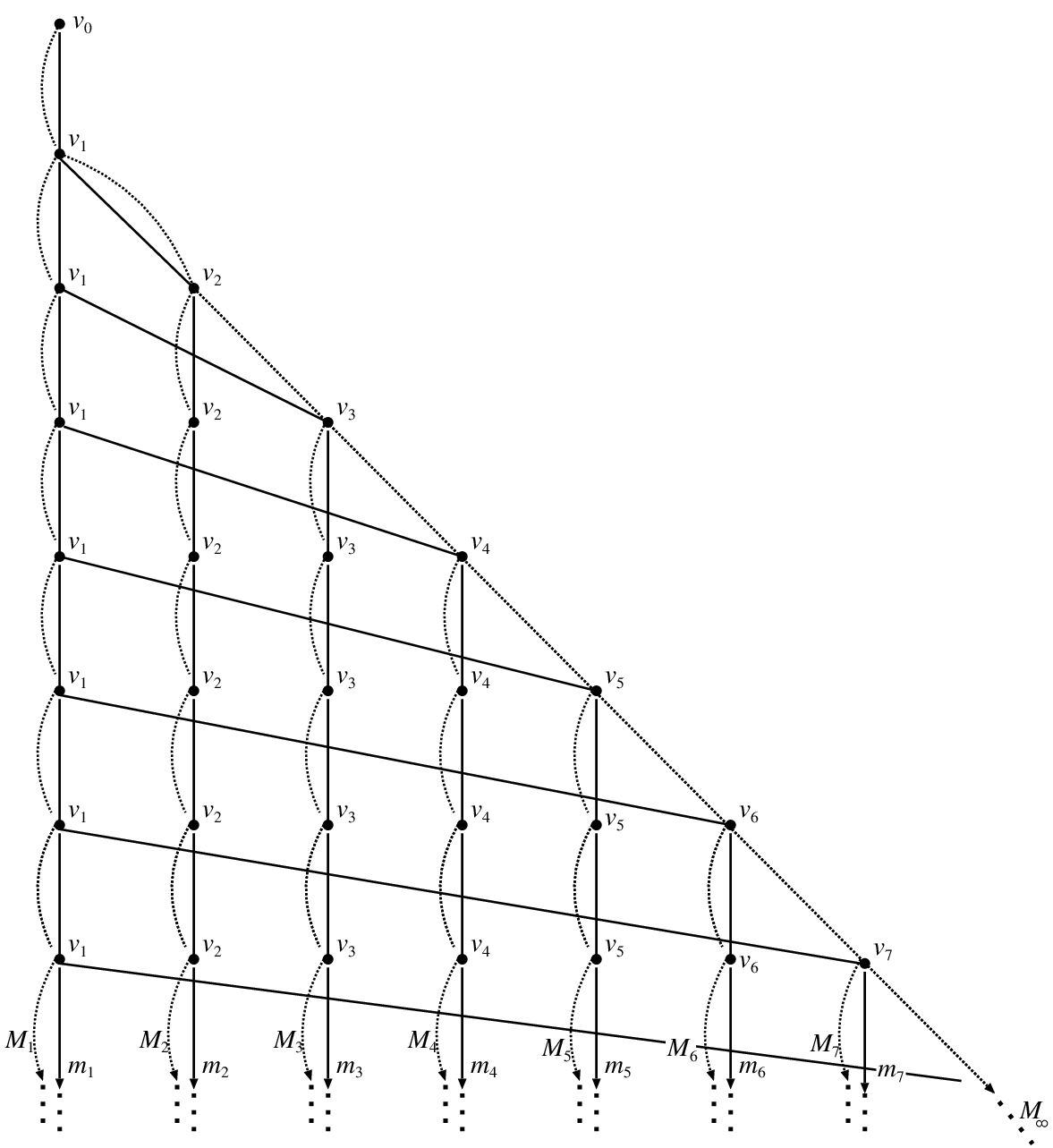}}
   \caption{The minimal and maximal edge structure for Example \ref{basic_infinite_rank_example}: solid lines are minimal edges, dashed lines are maximal.}
   \label{skeleton_example}
  \end{figure}

\begin{example}\label{basic_infinite_rank_example_uncountable}
Suppose that for $n\geq 1$, $V_{n}=\wt{V}_n= \ol {V}_n= \{v_{1}, \ldots , v_{2^n}\}$, and all incidence matrix entries are at least two. We label each vertex $v_i\in V_n$ using a binary string $(x_1, \ldots x_n)$ of length $n$ which denotes $i'$s binary expansion, starting with the least significant digit.
For example, we label vertex $v_5 \in V_4$ with the string $(1010)$. Suppose that  for $v_i \in V_{n+1}$,
\[ s(\wt e_{  (x_1, \ldots , x_{n+1})   })   = s(\ol e_{   (x_1, \ldots , x_{n+1})    }) =(x_1, \ldots , x_n). \]
In this case  the sets  $ \{  M_\alpha: \alpha \in \{0,1\}^{\mathbb N}    \}  $
and $ \{  m_\alpha: \alpha \in \{0,1\}^{\mathbb N}    \}  $ are uncountable.
\end{example}

Note that the previous  examples illustrate the fact that when defining a skeleton, we do not need complete information about the Bratteli diagram $B$. In particular, at this point, we need to know very little about the incidence matrices $(F_n)$ of $B$. The skeleton $\mathcal F$ is simply a constrained set of choices for all extremal edges when building an order.

Next we  discuss a way of building a homeomorphism $\sigma:X_{\max}(\mathcal F) \rightarrow X_{\min}(\mathcal F)$ which is amenable to being extended to a Vershik map. This will be the analogue of the permutation associated to a skeleton in the finite rank case.

Suppose that $\om$ is an order on $B$. Let $\sigma_{n}: \wt V_{n}\rightarrow 2^{ \ol  {V_{n}}}$ be defined by $\ol v\in \sigma_{n}(\wt v)$ if and only if
$\wt v \ol v \in \mathcal L(B, \om ,n)$.
 If $\om$ is perfect, then
  for any  sequence $(\wt v_n) = ( v_n(M_\alpha))$, there is a unique  sequence $(\ol v_n) = ( v_n (m_\beta))$ with  $\ol v_n \in \sigma_n (\wt v_n)$  for each $n$.
 If a perfect order $\om$ has a finite number of extremal paths, we can say more. In that case, for all large enough $n$, and all maximal $\wt v\in \wt V_n$, $ \sigma_n (\wt v) $ is an element, not a subset,  of  $\ol V_n $ - i.e. each maximal vertex can only be followed by a unique minimal vertex in  $\mathcal L(B,\om, n)$. One can then say that for all large $n$, $|\wt V_n|=|\ol V_n|$ and $\sigma_n:\wt V_n \rightarrow \ol V_n$ is a bijection. In the case of finite rank, the sets $\wt V_n$ and $\ol V_n$, and the maps $\sigma_n:\wt V_n \rightarrow \ol V_n$ could be taken to be equal for all $n$, and in that case we called $\sigma:=(\sigma_n)$ a {\em permutation} in \cite{bky}.

 In the general case of a perfect order with infinitely many extremal paths, though, this fact - that
 for any  sequence $(\wt v_n) = ( v_n(M_\alpha))$, there is a unique  sequence $(\ol v_n) = ( v_n (m_\beta))$ with  $\ol v_n \in \sigma_n (\wt v_n)$  for each $n$ -  does not generally imply that $(\wt v_n) = ( v_n(M_\alpha))$ is eventually a sequence of singletons. The main obstacle is that one can have pairs of distinct maximal paths that agree on an arbitrarily large initial segment. For,
suppose that $\varphi_\om(M) = m$, and assume that $M'$ is another maximal path that coincides with $M$ for the first $n$ segments till the vertex $\wt v_n$. By continuity of $\varphi_\om$, the minimal path $m' = \varphi_\om(M')$ must be close to $m$, but it can be that $\ol v_n =  v_n(m) \neq v_n(m') = \ol w_n$. So, we  see that not only  $\wt v_n\ol v_n \in L(B,\om, n)$ but also $\wt v_n\ol w_n \in L(B,\om, n)$, that is $\{\ol v_n, \ol w_n\} \in \sigma_n(\wt v_n)$.
One can  build such orders on any Bratteli diagram: see Example \ref{basic_set_map_correspondence_diagram}.

We use these observations to make the following definition, which generalizes the concept of a permutation for finite rank diagrams. Some notation: if $\sigma:\wt V \rightarrow 2^{\ol V}$  and   $\ol v \in \ol V$,  we define $\sigma^{-1} (\ol v):= \{\wt v: \ol v \in \sigma (\wt v)  \}.$

   \begin{definition}\label{correspondence_definition}
  Let $\mathcal F$ be a skeleton for an unordered  Bratteli diagram $B$. Suppose that  $\sigma = (\sigma_n)_n$ is a sequence of maps   $\sigma_n:\wt V_n \rightarrow 2^{\ol V_n}$, such that for each $n$,
   $\bigcup_{\wt  v\in \wt V_n} \sigma_n(\wt v ) = \ol V_n$,    and
  \begin{enumerate}
  \item $\sigma$ is {\em composition consistent:} let $M(n,N, \wt v)$ and $m(n,N, \ol v)$ denote the maximal  and minimal  paths from level $n$ to level $N  > n$ with range $\wt v$  and $\ol v$ respectively. If $\ol v \in \sigma_N(\wt v)$, then
  $s(m(n, N, \ol v)) \in  \sigma_n (s(M(n,N, \wt v)))$,
  \item
  for any $M\in X_{\max} (\mathcal F)$, there is a unique $m \in X_{\min}(\mathcal F)$ with $(v_n (m))_n
  \in \prod_n \sigma_n ( v_n (M))$,
  \item
   for any $m\in X_{\min} (\mathcal F)$, there is a unique $M \in X_{\max}(\mathcal F)$ with $(v_n (M))_n
  \in \prod_n \sigma_n^{-1} ( v_n (m))$, and
     \item  the bijection $\sigma: X_{\max} (\mathcal F) \rightarrow X_{\min} (\mathcal F)$ defined using properties 2 and 3  is a homeomorphism.
      \end{enumerate}
 Then we say that
$\sigma=(\sigma_{n})$ is a  {\em correspondence associated to $\mathcal F$}.
  \end{definition}

\begin{example}\label{basic_infinite_rank_correspondence}
We continue with the skeleton defined in Example \ref{basic_infinite_rank_example} and Figure \ref{skeleton_example}.
Let $\sigma$ be
 defined by
\[
\sigma_{n}(v_i)
   = \left\{
\begin{array}{rl}
\{v_{i+1}\}
& \mbox{if $1\leq i \leq n-1$}
   \\
\{v_{1} \}& \mbox{ if $i=n$}
  \end{array}
  \right. \,
\]
for each $n\geq 1$; then one can verify that $\sigma$ is composition consistent. Since each  $\sigma_n:\wt V_n \rightarrow \ol V_n$ is in fact a point map, this means that items (2) and (3) of Definition \ref{correspondence_definition} are satisfied. The homeomorphism $\sigma: X_{\max}(\mathcal F) \rightarrow X_{\min} (\mathcal F) $ satisfies $\sigma(M_{\infty} )= m_1$, and for $i\geq 1$, $\sigma(M_i) = m_{i+1}$.
\end{example}

\begin{example}\label{basic_infinite_rank_correspondence_uncountable}
We continue with the skeleton defined in Example \ref{basic_infinite_rank_example_uncountable}.
Let `+1' denote addition with carry, so that  for example, $(1010)+1 = (0110)$.
Let $\sigma = (\sigma_n)$ be
 defined by
\[
\sigma_{n}((x_1, \ldots  , x_n))
   = \left\{
\begin{array}{rl}
\{(x_1, \ldots , x_n)+1\}
& \mbox{if $(x_1, \ldots , x_n) \neq (1,\ldots 1)$}
   \\
\{(0,\ldots ,0) \}& \mbox{ if $(x_1, \ldots , x_n) = (1,\ldots 1)$}
  \end{array}
  \right. \,
\]
for each $n\geq 1$; then one can verify that $\sigma$ is composition consistent. Since each  $\sigma_n:\wt V_n \rightarrow \ol V_n$ is in fact a point map, this means that items (2) and (3) of Definition \ref{correspondence_definition} are satisfied.
Note that the bijection
$\sigma: X_{\max}(\mathcal F) \rightarrow X_{\min} (\mathcal F) $ satisfies $\sigma(M_{111 \ldots} )= m_{000\ldots}$, and  $\sigma(M_{x_1 x_2 \ldots}) = m_{y_1 y_2\ldots}$,  where $(y_1 y_2 \ldots ) = (x_1x_2 \ldots) +(100\ldots)$; in other words, $\sigma$ is the binary odometer map.
\end{example}

\begin{example}\label{basic_set_map_correspondence_example}
It seems to be  difficult to find examples of skeletons and accompanying correspondences where the maps $\sigma_n$ are not eventually point maps. If  $v_N(m)$ and $v_N(m')$ both belong to $\sigma_N(v_N(M))$,
the composition consistency condition forces $v_n(m)$ and $v_n(m')$  to belong to $\sigma_n(v_n(M))$ for $n<N$, making it  hard for points (2) and (3) of the definition of a correspondence to be satisfied. Here is one example, illustrated in  Figure \ref{basic_set_map_correspondence_diagram}. The vertex structure of this diagram is as in  Example \ref{basic_infinite_rank_example}: $V_n=\wt V_n= \ol V_n = \{v_1, \ldots ,v_n\}$. To define the skeleton, we let $s(\wt e_{v_i})= s(\ol e_{v_i})= v_{i-1}$ for $i>2$, $s(\wt e_{v_2})=v_1$, $s(\ol e_{v_2})=v_2$, and $s(\wt e_{v_1})=s(\ol e_{v_1})=v_1$. This skeleton is illustrated in Figure \ref{basic_set_map_correspondence_diagram}.  We define, for each $n$, $\sigma_n (v_1)= \{ v_2, v_3  \}$,  $\sigma_n (v_i) = v_{i+1} $ for $i=2, \ldots, n-1$, and $\sigma_n(v_n)=v_1$. Then $(\sigma_n)_n$ defines a correspondence, with $\sigma(M_n) = m_{n-1}$ for $n\geq 1$, and
$\sigma(M_0) = m_\infty$.
\end{example}

  \begin{figure}\label{basic_set_map_correspondence_diagram}
   \centerline{\includegraphics[scale=0.7]{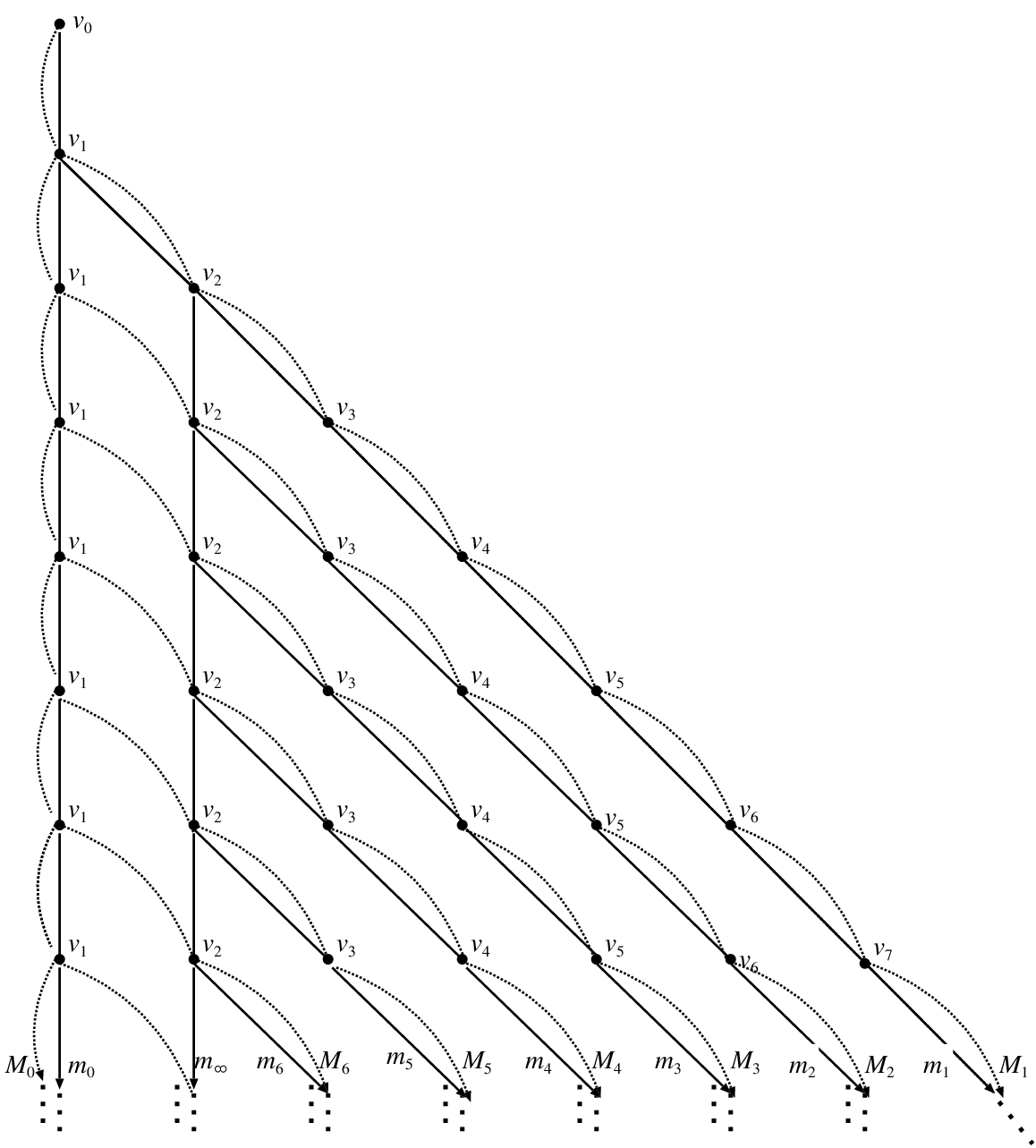}}
   \caption{The minimal and maximal edge structure for Example \ref{basic_set_map_correspondence_example}} 
  \end{figure}

In the case where  $\mathcal F$ is the skeleton associated to a well telescoped ordered diagram $(B,\om)$, we will {\em always} take  $\sigma_n$ to be that defined by $\mathcal L(B, \om, n)$, as discussed in the paragraph preceding Definition \ref{correspondence_definition}. Namely, given an order $\om$ with skeleton $\mathcal F$, we define $\ol  v\in \sigma_n (\wt v) $  if and only if $\wt v \ol v \in \mathcal L(B, \om, n)$.
Whether or not $\sigma=(\sigma_n)$ is a correspondence depends on whether $\om$ is perfect, as seen in the following proposition:

\begin{theorem}\label{infinite_rank_perfect}
Let $(B,\om)$ be a well telescoped ordered Bratteli diagram with skeleton $\mathcal F$ and accompanying maps $\sigma=(\sigma_n)$.
Then $\om$ is perfect if and only if $\sigma$ is a correspondence.
\end{theorem}

\begin{proof}
Suppose that $\om$ is perfect. The fact that $(\sigma_n)$ is composition consistent follows from the definition of the level $n$ languages $\mathcal L(B,\om,n)$. If it is the case that   for distinct minimal paths $m$, and $m'$, and a maximal path $M$, the two sequences $(v_n(m))$ and $(v_n(m'))$ belong to $(\sigma_n(v_n(M)))$ then we can build two sequences of paths $(x_n)$ and $(y_n)$, both converging to $M$, where $\varphi_\om (x_n) \rightarrow m$ and $\varphi_\om (y_n) \rightarrow m'$,  contradicting continuity of  $\varphi_\om$. Thus for each maximal path $M$, there is a unique
 $m \in X_{\min}(\om)$ with $(v_n (m))_n
  \in \prod_n \sigma_n ( v_n (M))$, and the continuity of $\varphi_\om$ implies that in fact $m=\varphi_\om (M)$.
  Similarly
   for any $m\in X_{\min} (\om)$, there is a unique $M = \varphi_\om^{-1}(m)  \in X_{\max}(\om)$ with $(v_n (M))_n
  \in \prod_n \sigma_n^{-1} ( v_n (m))$. Thus $\sigma:  X_{\max}(\om) \rightarrow X_{\min}(\om)$ coincides with $\varphi_\om:  X_{\max}(\om) \rightarrow X_{\min}(\om)$, and the fact that $\varphi_\om$ is a homeomorphism and $X_{\min}(\om)$, $X_{\max}(\om)$ are both closed  implies that $\sigma:  X_{\max}(\om) \rightarrow X_{\min}(\om)$ is a homeomorphism.

Conversely, suppose that $\sigma$ is a correspondence. The Vershik map $\varphi_\om$ is well defined everywhere, and continuous,  outside the sets of extreme paths. We use  $\sigma$ to define $\varphi_\om$ on $X_{\max}(\om)$, so that  $\sigma$ equals the restriction of $\varphi_\om$ to $X_{\max}(\om)$, a similar statement holding for $\sigma^{-1}.$
As $\sigma$ is a correspondence,   $\varphi_\om$ is continuous on $X_{\max}(\om)$; thus   to check continuity of $\varphi_\om$,
it is  sufficient to consider a  convergent sequence $(x_n)$ of  non-maximal paths, where $x_n\rightarrow M$ with $M$ maximal.  We claim that $(\varphi_\om (x_n))$  converges to some minimal sequence $m$ (and in fact this $m$ does not depend on the choice of $(x_n)$). Suppose not. Then for two subsequences $(y_n)$ and $(y_n')$ of $(x_n)$, we have $\varphi_\om (y_n)\rightarrow m$ and $\varphi_\om (y_n') \rightarrow m'$ for two paths $m\neq m'$, which are necessarily minimal paths.

Since each $y_n$  is not maximal, this implies that for some subsequence $(n_k)$, $v_{n_k}(m) \in \sigma_{n_k}(v_{n_k}(M))$, which implies, by composition consistency,  that $v_n(m) \in \sigma_n(v_n(M))$ for each $n\geq 1$. Similarly,  as each $y_n'$ is not maximal for some subsequence $(n_k')$, $v_{n_k'}(m') \in \sigma_{n_k'}(v_{n_k'}(M))$, which implies that $v_n(m') \in \sigma_n(v_n(M))$ for each $n$.
Since we have assumed that $\sigma$ is a correspondence, this contradicts the fact that there is a unique minimal element $m=\sigma(M)$ such that  $(v_n(m))_n \in \prod_n (\sigma_n (v_n(M)))$.  Compactness ensures the continuity of $\varphi_\om^{-1}.$

\end{proof}

Theorem \ref{infinite_rank_perfect} tells us that behind every perfect order $\om$ on a diagram $B$, there is an underlying skeleton $\mathcal F$ and correspondence $\sigma$. More generally, a correspondence accompanying a skeleton $\mathcal F$ will contain the information that allows us to extend the  partial definition of  orders  using $\mathcal F$ to  construct perfect orders. As in the finite rank case, the notions of accompanying partitions and associated directed graphs will be useful.

\begin{definition}
Suppose that $B$ is a  Bratteli diagram with skeleton $\mathcal F = (\wt V_{n-1}, \ol V_{n-1}, \{\wt e_{v}, \ol e_{v}: v\in V_{n}\}: n \geq 2)$ associated to  the set $\{M_\alpha: \alpha \in I\}$ of maximal paths and the set $\{m_\beta : \beta \in J\}$ of minimal paths.
For any vertices $\wt v \in
\wt V_{n-1}$ and $\ol v \in \ol V_{n-1}$, we set
\begin{equation}\label{W} W_{\wt v}(n) = \{w \in V_n : s(\wt e_w) =
\wt v\}, \ \ W'_{\ol v}(n) = \{w \in V_n : s(\ol e_w) =
\ol v\}
\end{equation}
where $ n\geq 2$. It is obvious that $W(n) = \{W_{\wt v}(n) : \wt v
\in \wt V_{n-1}\}$ and $W'(n) = \{W'_{\ol v}(n) : \ol v \in \ol
V_{n-1}\}$ form two partitions of $V_n$.
We call the sequence of partitions $W = (W(n))_n$ and $W'=(W'(n))_n$ the {\em partitions generated by  $\mathcal F$.}
\end{definition}

The intersection of $W(n)$
and $W'(n)$ is the partition $W'(n)\cap W(n)$ whose elements are
non-empty sets $W'_{\ol v}(n) \cap W_{\wt v}(n)$ where $(\ol v, \wt v)
\in \ol V_{n-1} \times \wt V_{n-1}$. We shall use the notation $[\ol v,\wt v, n]:=
W'_{\ol v}(n)\cap W_{\wt v}(n)$ for shorthand.

\begin{definition}\label{infinite_associated_directed_graphs_definition}
Let $\mathcal F$  be a skeleton on $B$ with
an associated correspondence $\sigma$.
Let  $\mathcal H_n=(T_n,P_n)$ be the directed graph where
the set $T_{n}$ of vertices of  $\mathcal H_{n}$ will consist of partition elements $[\ol v, \wt v, n]$ of
$W(n) \cap
W'(n)$, and where there is an edge in  $P_{n}$ from $[\ol v, \wt v, n]$ to $[\ol v', \wt v', n]$
if and only if $\ol v' \, \in \sigma_{n-1}(\wt v)$.
We call $(\mathcal H_n)$ the {\em sequence of directed graphs associated to $(\mathcal F, \sigma)$}.
\end{definition}

\begin{remark}The vertices of $\mathcal H_n$ are labeled by $[\ol v, \wt v, n]$ where $(\ol v, \wt v)\in \ol V_{n-1} \times \wt V_{n-1}$. On the other hand, the set $[\ol v, \wt v, n]$ is a set of  vertices in $V_n$. When we speak about a path in $\mathcal H_n$, we mean a concatenated sequence of directed edges between vertices of $\mathcal H_n$; these paths  will correspond to families of words in $V_n^+$. The next proposition tells us that if $\om$  is  perfect, then words in $\mathcal L(B,\om,n)$ must come from a valid path in $\mathcal H_n$;
it is the appropriate generalization of Lemma \ref{graphs_determine_perfection}.
\end{remark}

\begin{proposition}\label{perfect_order_characterisation}
Suppose that $\mathcal F$ is a skeleton on $B$, with a correspondence $\sigma$. Let $(\mathcal H_n)_n$ be the sequence of directed graphs associated to $(\mathcal F,\sigma)$.

\begin{enumerate}
\item
If  the perfect order $\om$ has associated skeleton and correspondence $(\mathcal F, \sigma)$,  then words in $\mathcal L(B,\om,n)$ correspond to paths in $\mathcal H_n$.
\item
Let $\om$  be defined to have $\mathcal F$ and correspondence $\sigma$, and where for each $n$, all words in $\mathcal L(B,\om,n)$ correspond to paths in   $\mathcal H_n$. Then $\om $  is perfect.
\end{enumerate}
\end{proposition}

\begin{proof}

For a given perfect order $\om$, the map $\sigma_n$ is defined using the language $\mathcal L(B,\om,n)$. If $vw\in \mathcal L(B,\om,n)$, where $v\in [\ol v, \wt v, n]$ and $w\in [\ol v', \wt v',n]$, then $\wt v \ol v' \in \mathcal L(B,\om, n-1)$, so that $\ol v' \in \sigma_{n-1} (\wt v)$. Thus $vw$ corresponds to a path in $\mathcal H_n$. The argument for longer words in $\mathcal L(B,\om,n)$
is similar.

To prove the second statement, take non-maximal paths $(x_n)$ converging to a maximal  $M$. We shall show that in fact $\varphi_\om (x_n) \rightarrow \sigma(M)$. This implies that $\varphi_\om$ is continuous, and also that $\varphi_\om:X_{\max}(\om) \rightarrow X_{\min}(\om)$ can be defined coinciding with $\sigma:X_{\max}(\mathcal F) \rightarrow X_{\min}(\mathcal F)$.
Suppose that $x_n$ agrees with $M$ to level $k_n$.
Then, since $\sigma$ is composition consistent,
$v_j(\varphi_\om(x_n))\in \sigma_j (v_j(M))$ for each $j\leq k_n$.   For some subsequence $n_l$,  $\varphi_\om (x_{n_l})\rightarrow m$ where $m$ is a minimal path. This implies that $(v_j(m)) \in \prod_{j=1}^{\infty} \sigma(v_j(M))$, and by conditions (2) and (4) of Definition \ref{correspondence_definition}, $m = \sigma(M)$. Since any subsequence of $(\varphi_\om (x_n))$ has a subsequence that converges to $m$, it follows that $\varphi_\om (x_n) \rightarrow m$.
\end{proof}

\begin{example}
\label{basic_infinite_rank_graph}
We continue Examples \ref{basic_infinite_rank_example} and \ref{basic_infinite_rank_correspondence}:
illustrated in Figure \ref{H_n_example} is the graph $\mathcal H_n$ associated to the skeleton and correspondence considered in those examples.
Let $w^{(n)}= v_{1}\ldots v_{n}$; then $w^{(n)}$ is generated from a
 path in $\mathcal H_{n}$.
  Suppose that all words $w(v,n,n+1)$ are  defined using $w^{(n)}$, subject to the constraints of the skeleton $\mathcal F$ defined in Example \ref{basic_infinite_rank_example}, for example, $w(v_1, n, n+1)$ must both start and end with $v_1$. Then, {\em provided} that the incidence matrices $(\mathcal F_n)$ of $B$ allow us, we can define a perfect order $\om$ with skeleton $\mathcal F$ as in Example \ref{basic_infinite_rank_example} and accompanying correspondence $\sigma$ as in Example \ref{basic_infinite_rank_correspondence}.  For example, the $v_1$-indexed row of $\mathcal F_n$ must be of the form $(\alpha_n +1, \alpha_n, \ldots, \alpha_n, \beta_n)$ where $\alpha_n$ and $\beta_n$ are  positive integers. This will be further elucidated in Theorem \ref{analogous condition}.
\end{example}

 \begin{figure}[h]
\centerline{\includegraphics[scale=0.9]{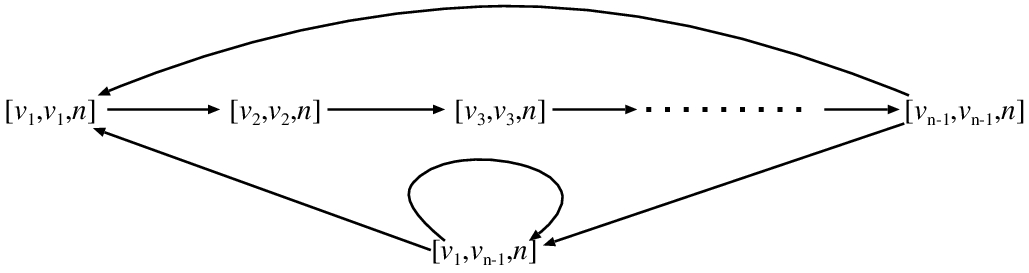}}
\caption{The graph $\mathcal H_n$ in Example \ref{basic_infinite_rank_graph}}\label{H_n_example}
 \end{figure}

Finally we state and prove the analogue of Proposition 3.19 in \cite{bky}. Recall that a directed graph is {\em strongly connected} if for any two vertices $v$, $v'$, there are paths from $v$ to $v'$, and also from $v'$ to $v$. If at least one of these paths exist, then $G$ is {\em weakly connected}. Recall also the definition of the family $\mathcal A$ of Bratteli diagrams (Definition \ref{aperiodic_definition}).

\begin{proposition}\label{connectedness and goodness}
 Let $(B, \om)$ be a well telescoped, perfectly  ordered  Bratteli diagram
with skeleton $\mathcal F_\om$,
and correspondence $\sigma$. Let  $(\mathcal H_n)_n$ be the sequence of associated directed graphs.

\begin{enumerate}
\item If $B$ is simple, then  $\mathcal H_n$ is
strongly connected for any $n$.
\item If $B \in \mathcal A$,
 $\mathcal H_n$ is
weakly connected for any $n$.
\end{enumerate}
\end{proposition}

\begin{proof} We prove (1) - the proof of (2) is similar (if we focus on $w(v,n-1,n)$ where $v$ is the vertex which indexes the strictly positive row in $F_{n}$).  In the case of simple diagrams, we can assume that all entries of $F_n$ are positive for each $n$.

 Take two vertices $[\ol v_1, \wt v_1,n]$ and $ [\ol v_3, \wt v_3 ,n]$
in $\mathcal H_n$. If $\ol v_2\in \sigma_{n-1}(\wt v_1)$, then there is some vertex $\ol{\ol v}\in \ol V_n$ such that $s(\ol e_{\ol{\ol v}})=\ol v_2$. Let $\ol{\ol v}\in [\ol v_2, \wt v_2,n]$. Clearly there is an edge from
$[\ol v_1, \wt v_1,n]$ to $[\ol v_2, \wt v_2, n]$ in $\mathcal H_n$.

 Let $v\in V_{n+1}$ be such that $s(\ol e_v)= \ol{\ol v}$.
Let $v_3\in [\ol v_3,\wt v_3,n]$. Since $B$ is simple, $f_{v,v_3}^{(n)}>0$; this means that $\ol{\ol  v} \ldots v_3$ is a prefix of $w(v,n,n+1)$. This implies that there is  path in $\mathcal H_n$ from  $[\ol v_2, \wt v_2,n]$ to $[\ol v_3,\wt v_3,n]$.
\end{proof}

\begin{remark} It is not hard to see that the converse statement to Proposition \ref{connectedness and goodness} is not true. There are examples of non-simple diagrams of finite rank whose associated graphs are strongly connected.

Note also that the
assumption that $\om$ is perfect is crucial.
Moreover, there are examples of \textit{simple} finite rank Bratteli
 diagrams and skeletons none of whose associated graphs are  strongly
connected. Indeed,
 let $B$ be a stationary  diagram with $V = \{a,b,c\}$ with the skeleton
 $\mathcal F = \{M_a, M_b, m_a, m_b; \wt e_c, \ol e_c\}$ where $s(\wt
 e_c) =b, s(\ol e_c) = a$. Let $\sigma(a) = a, \sigma(b) =
 b$. Constructing the associated graph $\mathcal H$, we see that there
 is no path  from $[b,b]$ to $[a,a]$. It can be also shown that there is
 no perfect ordering $\om$ such that $\mathcal F = \mathcal F_\om$. This
 observation complements  Proposition \ref{connectedness and goodness}
 by stressing the importance of the strong connectedness of $\mathcal H_{n}$ for the existence of perfect orderings.
\end{remark}

\section{Characterizing Bratteli diagrams that support perfect orders}\label{characterization}

In this section we characterize Bratteli diagrams that support perfect, non-proper orders via their incidence matrices. Our main result is Theorem    \ref{analogous condition}, which extends a similar result proved in \cite[Theorem 4.6]{bky} for finite rank diagrams. We  define the class $\mathcal P_B^*$, a set of perfect orders whose language properties are similar to those of orders in $\mathcal P_B(j)$, $j$ finite,  and for whom a refined version of Theorem   \ref{analogous condition} holds, namely  Corollary \ref{analogous_special}.

The intuition behind the proof of Theorem \ref{analogous condition} is
the following idea.  If a diagram $B$ is to support a perfect,
well-telescoped order $\om$, then $\om$ would define a skeleton $\mathcal
F$ and correspondence $\sigma$. The correspondence intrinsically
contains the information about the languages defined by $\om$, and
this is further expressed with the sequence $(\mathcal H_n)$ of
directed graphs, in that words in $\mathcal L(B,\om,n)$ must
correspond to paths in $\mathcal H_n$.   Words in $\mathcal L(B,\om,n)$
are generated by the orders placed on the edges in $r^{-1}(u)$, where
$u\in V_m$ and $m>n$.  To define $w(u,n,n+1)$ for $u\in V_{n+1}$, then, we use $\mathcal H_n$.   The edge structure of $\mathcal
H_n$ implies that if a word $vw$ lies in $w(u,n,n+1)$ with
$v$ belonging to $W_{\wt v}(n)$, then $w$ must belong to $W_{\ol
  v}(n)$ for some $\ol v \in \sigma_{n-1}(\wt v)$ - this is the
statement of Lemma \ref{successor}. This means that every time we leave a vertex in $\mathcal H_n$ of the form $[*,\wt v , n]$, we must go to a vertex of the form $[\ol  v ,*,  n]$ for some $\ol v \in \sigma_{n-1}(\wt v)$.  Thus the $u$-th row of the $n$-th incidence
matrix $F_{n}$ must have a `balance' between entries
$f_{u,v}^{(n)}$, where $v\in W_{\wt v}(n)$, and entries
$f_{u,v'}^{(n)}$, where $v'\in \bigcup_{ \ol v \in \sigma_{n-1} (\wt
  v) } W_{\ol v}(n)$. This is more precisely stated in Corollaries
\ref{necessary condition general} and \ref{necessary condition}. It
turns out that these `balance' requirements, the system of relations
(\ref{general equal cardinalities one a}), along with the system
(\ref{general equal cardinalities two a}), are also sufficient for the existence
of a perfect order on $B$.

First, we define the class $\mathcal P_B^*$, a class of perfect orders that  naturally generalizes the class of perfect orders with finitely many extremal paths,  and also  introduce/re-introduce notation we shall need.

\begin{definition} Let $(B,\om)$ be a well telescoped, perfectly  ordered diagram  with skeleton $\mathcal F$ and correspondence $\sigma =(\sigma_n)$.
We say that $\om$ belongs to $\mathcal P_B^*$ if
$\om$ satisfies the following conditions:
 for each maximal path $M$ with $\wt v_n = v_n(M)$, $\sigma_n(\wt v_n) \in \ol V_n$   for all $n$ sufficiently large, and for each minimal path $m$ with $\ol v_n = v_n(m)$, $\sigma_n(\ol v_n)^{-1} \in \wt V_n$  for all $n$ sufficiently large.
 \end{definition}
In fact, it seems (see the comment in  Example  \ref{basic_set_map_correspondence_example}) unnatural for a perfect order to not belong to $\mathcal P_B^*$.  All well telescoped perfect orders with finitely many maximal paths belong to $\mathcal P_B^*$, i.e. $\mathcal P_B\cap \mathcal O_B(j) \subset \mathcal P_B^*$ for each finite $j$.
Also, suppose the perfect  well telescoped order $\om$ is such that  for each maximal  path $M_\alpha$, and each minimal path $m_\beta$, there exist neighborhoods $U(M_\alpha)$ and $U(m_\beta)$ such that no other maximal paths belong to $U(M_\alpha)$, and no other minimal paths are in $U(m_\beta)$; then $\om \in \mathcal P_B^*$.  Note that such an order can have at most countably many extremal paths, since
  the correspondence $M_\alpha \to U(M_\alpha)$ is injective, and  the set of clopen sets is countable.

Let $\om $ be a perfect order on $B$. Recall  that $\om$ generates the skeleton $\mathcal F_\om = (\wt V_{n-1}, \ol V_{n-1}, \{\wt e_v, \ol e_v: v\in V_n\}, n >1)$ and two partitions $W(n) = \{W_{\wt v}(n) : \wt v \in \wt V_{n-1}\}$ and $W'(n) = \{W'_{\ol v}(n) : \ol v \in \ol V_{n-1}\}$ of $V_n$. Moreover, we have also a sequence of correspondences $\sigma_n : \wt V_n \to 2^{\ol V_n}, n \geq 1,$ defined by $\om$. We recall also the notation used for maximal (minimal) paths: if $M$ is a maximal path then it determines uniquely a sequence of maximal vertices $(\wt v_n = v_n(M))$.

Let $E(V_n, u)$ be the set of all finite paths
between vertices of level $n$ and a vertex $u\in V_m$ where $m >
n$. The symbols $\wt e(V_n, u)$ and $\ol e(V_n, u)$ are used to denote
the maximal and minimal finite paths in $E(V_n, u)$, respectively; if $m = n+1$ so that $u\in V_{n+1}$, then we revert to the shorter notation $\wt e_u$ and $\ol e_u$. Fix a
maximal and minimal vertex $\wt v \in \wt V_{n-1}$ and $\ol v\in \ol V_{n-1}$ respectively. Denote $E(W_{\wt v}(n), u) = \{e \in E(V_n, u): s(e) \in W_{\wt v}(n), r(e) = u\}$ and
$E(W'_{\ol v}(n), u) = \{e \in E(V_n, u): s(e) \in W'_{\ol v}(n), r(e) = u\}$.
Clearly, the sets $\{E(W_{\wt v}(n), u): \wt v \in \wt V\}$ and $\{E(W_{\ol v}(n), u): \ol v \in \ol V\}$ form two partitions of $E(V_n, u)$.
It may happen that the maximal finite path $\wt e(V_n, u)$ has its source in $W_{\wt v}(n)$. In this case, we define $\wt E(W_{\wt v}(n), u) = E(W_{\wt v}(n), u) \setminus \{\wt e(V_n, u)\}$. Otherwise,  $\wt E(W_{\wt v}(n), u) = E(W_{\wt v}(n), u)$. Similarly we define the set $\ol E(W'_{\ol v}(n), u)$ using the minimal finite path $\ol e(V_n, u)$.

In the next few results we assume that a perfect, well telescoped  order $\om$ has attached its skeleton $\mathcal F$, correspondence $\sigma$ and partitions $(W(n))$ and $(W'(n))$. Also, for brevity we shall abuse notation: if $e$ is an edge, (or a set of edges), we will write $\varphi_\om(e)$ instead of the more correct $\varphi_\om (U(e))$.

\begin{lemma}\label{successor} Suppose $(B, \om)$ is a well telescoped, perfectly  ordered Bratteli diagram. Let $\wt v \in \wt V_{n-1}$. If  $u\in V_m$, $m>n$, and $e \in \wt E(W_{\wt v}(n),u)$, then $\varphi_\om(e) \in \bigcup_{ \ol v \in   \sigma_{n-1}(\wt v)   }\ol E(W_{\ol v}(n),u)$.
\end{lemma}

\begin{proof}
Extend $e$ to a path $e*$ in $\wt E(\wt v, u)$ by concatenating the maximal edge in $E(\wt v, s(e))$ to $e$. Similarly,
extend $\varphi_\om(e)$ to a path $e^{**} $ in $\ol E(\ol v,u)$ by concatenating the minimal edge in $E(\ol v, \varphi_\om(e))$ to $\varphi_\om(e)$.
Since $\varphi_{\om}(e^*)= (e^{**})$, this means that  $\wt v\ol v \subset w(u,n-1,m)$, so that $\wt v\ol v \in \mathcal L(B,\om,n-1)$. By definition of $\sigma_{n-1}$, $\ol v\in \sigma_{n-1}(\wt v)$.

\end{proof}

The following corollary  can be easily deduced  from Lemma \ref{successor}.

\begin{corollary}\label{necessary condition general}
Let $(B, \om)$ be a   well telescoped, perfectly ordered Bratteli diagram. Then  for any $n\geq 2$, $\wt v \in \wt V_{n-1}$
and $u\in V_m,$ $m>n$, we have
\begin{equation}\label{general equal cardinalities two}
\sum_{\ol v \in \sigma_{n-1}(\wt v)}
|  \wt E(W_{\wt v}(n), u) \cap \varphi_\om^{-1}( \ol E(W'_{\ol v}(n), u)     )    | = |  \wt E(W_{\wt v}(n), u)    |.
\end{equation}
Also if $\ol v\in \ol V_{n-1}$, then
\begin{equation}\label{general equal cardinalities one}
\sum_{\wt v: \ol v \in \sigma_{n-1}(\wt v)}
|  \wt E(W_{\wt v}(n), u) \cap \varphi_\om^{-1}( \ol E(W'_{\ol v}(n), u)     )    | = |  \ol E(W'_{\ol v}(n), u)    |.
\end{equation}
\end{corollary}

We can refine the statement of Corollary \ref{necessary condition general} in some special cases:

\begin{corollary}\label{necessary condition}
Let $(B, \om)$ be a well telescoped, perfectly  ordered Bratteli diagram.
\begin{enumerate}
 \item Suppose that  $\om \in \mathcal P_B \cap \mathcal O_B(j)$.  Then  there exists an $n_0$ such that for any $n \geq
n_0$, any vertex $\wt v \in \wt V_{n-1}$, any $m >n$, and any $u \in V_m$,  $\sigma_{n-1}(\wt v)$ is a singleton and one has
\begin{equation}\label{equal cardinalities}
|\wt E(W_{\wt v}(n), u)| = |\ol E(W'_{\sigma_{n-1}(\wt v)}(n), u)|.
\end{equation}
\item
Suppose that  $\om\in \mathcal P_B^*$.  Then  for any maximal path $M$ (and hence any  sequence $(\wt v_n = v_n(M))$), there exists an $n_0$ such that for any $n \geq n_0$,   $\sigma_{n}(\wt v_{n})$ is a singleton, and for
any $m >n> n_0$ and $u \in V_m$, one has
\begin{equation}\label{equal cardinalities 2}
|\wt E(W_{\wt v_{n-1}}(n), u)| = |\ol E(W'_{\sigma_{n-1}(\wt v_{n-1})}(n), u)|.
\end{equation}
\end{enumerate}
\end{corollary}

Given the incidence matrices $(F_n) $ for $B$,
where $F_n = \{(f_{u,w}^{(n)}):   u \in V_{n+1}, w \in V_{n}\},$ we
 define the sequences of modified incidence matrices $(\wt F_n)$ and $(\ol F_n)$ as in Section 4 of  \cite{bky}.
Namely, define  $\wt F_n = (\wt f_{u,w}^{(n)})$ and $\ol F_n = (\ol f_{u,w}^{(n)})$ by the following rule (here $w\in V_{n}, \ u\in V_{n+1}$ and $  n\geq 1$):
\begin{equation}\label{wt f}
\wt f_{u,w}^{(n)} = \left\{
                      \begin{array}{ll}
                        f_{u,w}^{(n)} - 1 &  \mbox{ if } \wt e_u\in E(w, u) \mbox{ and }\\
                        f_{u,w}^{(n)}  & \mbox{otherwise},
                      \end{array}
                    \right.
\end{equation}
and
\begin{equation}\label{ol f}
\ol f_{u,w}^{(n)} = \left\{
                      \begin{array}{ll}
                        f_{u,w}^{(n)} - 1 &  \mbox{ if }\ol e_u \in E(w, u) \mbox{ and } \\
                        f_{u,w}^{(n)}, & \mbox{otherwise}.
                      \end{array}
                    \right.
\end{equation}

Relation  (\ref{general equal cardinalities two}) implies that  for each $u\in V_{n+1}$ and each $\wt v \in \wt V_{n-1}$, if $w\in W_{\wt v}(n)$, then
each $\wt f_{u,w}^{(n)}$ can be written as
\begin{equation}
\label{general equal cardinalities two a}
\wt f_{u,w}^{(n)}= \sum_{\ol v\in \sigma_{n-1}(\wt v)} \wt f_{u,w,\ol v}^{(n)}
\end{equation}

where  if if $w\in W_{\wt v}(n)$, we define $\wt f_{u,w,\ol v}^{(n)}$ to be the number of non-maximal edges with source $w$, range $u$, and whose successor edge belongs to $\ol v\in \sigma_{n-1}(\wt v)$.

Relation (\ref{general equal cardinalities one}) says that for each $u \in V_{n+1}$ and  each $\ol v\in \ol V_{n-1}$,

\begin{equation}\label{general equal cardinalities one a}
\sum_{ \wt v: \ol v \in \sigma_{n-1}(\wt v)      }    \sum_{w\in W_{\wt v}(n)}\wt f_{u,w,\ol v}^{(n)}= \sum_{ w'\in W_{\ol v }'(n)}  \ol f_{u,w'}^{(n)} \, .
\end{equation}

We will  refer to the relations in (\ref{general equal cardinalities one a}) above as the {\em balance relations}.
 If $\om $ is perfect and has finitely many extremal paths, then the balance relations have the following form: for $n>n_0$, $\wt v \in \wt V_{n-1}$ and  $u \in V_{n+1}$:
\begin{equation}\label{sums of entries}
\sum_{w \in W_{\wt v}(n) }\wt f_{u,w}^{(n)} = \sum_{w'\in W'_{\sigma_{n-1}(\wt v)}(n)} \ol f _{u, w'}^{(n)}, \ \ \ u\in V_{n+1}.
\end{equation}
If $\om\in \mathcal P_B^*$ and the maximal path $M$ is given, then there exists $n_0$ such that for each $n>n_0$, if $v_{n-1}(M)=\wt v$ and $u\in V_{n+1}$, then relation (\ref{sums of entries}) is satisfied.

The content of  Theorem \ref{analogous condition}  is that given a skeleton and correspondence on $B$,
 relations (\ref{general equal cardinalities two a}) and (\ref{general equal cardinalities one a}) are sufficient conditions on the incidence matrices of a Bratteli diagram, in order that it supports a perfect order $\om$.
Our proof is constructive in that given a diagram, skeleton and correspondence,we use an algorithm to define, for each $u\in V_{n+1}$ and $n\in \N$, the word $w(u,n,n+1)$ - i.e. we order the set $r^{-1}(u)$. We do this by constructing a path in $\mathcal H_n$ that starts in $[\ol v_0, \wt v_0, n]$, where
$s(\ol e_u) \in [\ol v_0, \wt v_0, n]$, terminates at $[\ol v_t, \wt v_t, n]$, where $s(\wt e_u) \in [\ol v_t, \wt v_t, n]$, and passes through each vertex in $\mathcal H_n$ a prescribed number of times that we now make precise.

 Proposition \ref{connectedness and goodness} tells us that we have to assume that the directed graphs $\mathcal H_n$ are strongly connected. We make clear what we mean by this as follows.
Fix $n\in \N$ and $u\in V_{n+1}$. If
  $[\ol v, \wt v,n] \in \mathcal H_n$, we associate a number $P_u([\ol
    v, \wt v,n]):=\sum_{w\in[\ol v,\wt v, n]}\wt f_{u,w}^{(n)}$ to the
  vertex $[\ol v ,\wt v,n]$. This {\em crossing number} represents the
  number of times that we will have to pass through the vertex $[\ol
    v,\wt v,n]$ when we define an order on $r^{-1}(u)$, and here we
  emphasize that if we terminate at $[\ol v,\wt v,n]$ , we do not
  consider this final visit as contributing to the crossing number -
  this is why we use the terms $\wt f_{u,w}^{(n)}$, and not
  $f_{u,w}^{(n)}$. We say that $\mathcal H_n$ is {\em positively strongly
    connected}
    if for each $u\in V_{n+1}$, the set of vertices $\{ [\ol v,
    \wt v,n]: P_u([\ol v,\wt v,n])>0 \}$, along with all the relevant
  edges of $\mathcal H_n$, form a strongly connected subgraph of
  $\mathcal H_n$. If $s(\wt e_u) \in [\ol v, \wt v, n]$ we shall call this
  vertex in $\mathcal H_n$ the {\em terminal vertex}, as when defining
  the order on $r^{-1}(u)$, we need a path that ends at this vertex
  (although it can obviously go through this vertex several times - in
  fact precisely $P_u([\ol v, \wt v,n])$ times).

\begin{example} In this example we drop the dependence on $n$ and consider the stationary diagram $B = (V,E)$ that was used above in Example \ref{perfect order by telescoping}.
Suppose that $V=\{a,b,c,d\}$, $\ol V = \wt V = \{a,b,c\}$, with $a \in [a,a]$, $b\in [b,b]$, $c\in [c,c]$ and $d\in [b,a]$. Let $\sigma(a)=b$, $\sigma(b)=c$ and $\sigma(c) =a$. Suppose that the incidence matrix $F$ of $B$ is
\[F:=
\left(
\begin{array}{cccc}
2& 1 & 1 &  1 \\
1& 2 & 1  & 1 \\
1 & 1 & 2  & 1 \\
1  & 1  &1 & 2  \\
 \end{array}
\right)
\]
Then if $u=d$, $P_d([a,a])=0$, and the remaining three vertices $[b,b]$, $[c,c]$ and $[b,a]$ do not form a strongly connected subgraph of $\mathcal H$, - for example there is no path from $[c,c]$ to $[b,a]$. Hence for us $\mathcal H$ is not positively strongly connected.

Note also that although the rows of this incidence matrix satisfy the balance relations, there is no way to define an order on $r^{-1}(d)$ so that the resulting global order is perfect. The lack of positive strong connectivity of the graph $\mathcal H$ is precisely the impediment.
\end{example}

The following theorem is our main result,  extending a similar result (Theorem 4.6) in  \cite{bky}.

\begin{theorem}\label{analogous condition}
Let $B$ be a Bratteli diagram with incidence matrices $(F_n)$.  Let $\mathcal F$ be a skeleton on $B$ and $\sigma$ an associated correspondence, such that the graphs $\mathcal H_n$ are all positively strongly connected.
Suppose that the collection of natural numbers
\begin{equation} \label{decomposition}
\{ \wt f_{u,w, \ol v }^{(n)}: u \in V_{n+1}, w \in W_{\wt v}(n),  \ol v \in \sigma_{n-1} (\wt v) \}, \ \ \wt v \in \wt V_{n-1},
\end{equation}
 is such that relations (\ref{general equal cardinalities two a}) and
(\ref{general equal cardinalities one a}) are satisfied.

Then there exists a perfect order $\om$ on $B$ having $\mathcal F$ and $\sigma$ as associated skeleton and correspondence, respectively.
Conversely, suppose that a perfect $\om$ has accompanying skeleton and correspondence
$(\mathcal F, \sigma)$. Then there exists a set of natural numbers as in (\ref{decomposition}) such that relations (\ref{general equal cardinalities two a}) and      (\ref{general equal cardinalities one a})  hold.
\end{theorem}

\begin{proof}
As the preceding discussion deals with the necessity a perfect order having to satisfy  relations (\ref{general equal cardinalities two a}) and    (\ref{general equal cardinalities one a}), we prove here only the sufficiency of these relations.

  Our goal is to define a linear order on $r^{-1}(u)$ for each $u\in V_{n+1}$ and $n> 1$ - in other words to define $w(u,n,n+1)$ - so that the corresponding partial ordering $\om$ on $B$ is perfect. Recall that each set  $r^{-1}(u)$ contains two pre-selected edges $\wt e_u, \ol e_u$ and they should be the maximal and minimal edges in $r^{-1}(u)$ after defining $w(u,n,n+1)$.

Our proof is based on an inductive procedure that is applied to each row of the incidence matrices. We first describe in details the first step of the procedure that will be applied repeatedly. It will be seen from our proof that for given $B$, $\mathcal F$ and $\sigma$,  neither is the word  $w(u,n,n+1)$ that we define   unique, nor  will our procedure give all possible valid words.

We will first consider the particular case when the associated graphs $\mathcal H = (\mathcal H_n)$ defined by $\mathcal F$ do not have loops. After that, we will modify the construction to include possible loops.

 Case I: There are no  loops in the graphs $\mathcal H_n$. To begin with, we take some
 $u\in V_{n+1}$ and consider the $u$-th rows of matrices $\ol F_n$ and
 $\wt F_n$. They coincide with the row  $(f_{u,v_1}^{(n)},..., f_{u, v_d}^{(n)})$ of the matrix  $F_n$ except only one entry either corresponding to $|E(s(\ol e_u),
 u)|$ and $|E(s(\wt e_u), u)|$ in $\ol F_n$ and
 $\wt F_n$, respectively.  Take $\ol e_u$ and
 assign the number $0$ to it, i.e. $\ol e_u$ is the minimal edge in
 $r^{-1}(u)$. Let $[\ol v_0, \wt v_0,n] $ be the vertex\footnote{The same word
 `vertex' is used in two meanings: for elements of the set $T_n$ of the
 graph $\mathcal H_n$ and for elements of the set $V_n$ of the Bratteli
 diagram $B$. To avoid any possible confusion, we point out explicitly
 what vertex is meant in the context.} of $\mathcal H_n$ such that
 $s(\ol e_u) \in [\ol v_0, \wt v_0, n]$. Consider the set
\[\{ \wt f_{u,w}^{(n)} : w \in [\ol v, \wt v, n]: \ol v\in \sigma_{n-1}(\wt v_0) \}\, ,\]
and let $\wt f_{u,w'}^{(n)}$ be the maximum of this set,
 where $w'\in [\ol v_1,  \wt v_1 , n]$.
If there are several
 entries that are the maximal value, we chose one arbitrarily
 amongst them. Take any edge $e_1 \in E(w', u)$. In the case where  $\wt e_u \in
 E(w', u)$, we choose $e_1 \neq \wt e_u$. Assign the number $1$ to $e_1$
 so that $e_1$ becomes the successor of $e_0= \ol e_u$.

Two edges were labeled in the above
procedure, $e_0$ and $e_1$. We may think of this step as if these
edges were `removed' from the set of all edges in $r^{-1}(u)$.
In the collection of relations (\ref{general equal cardinalities one a})
we have worked with the relation
defined by $u$ and $\ol v_1$. On the left hand side,
the entry $\wt f_{u,s(\ol e_u), \ol v_1}^{(n)}$ was reduced by 1, and on the right hand side, $\ol f_{u,w'}^{(n)}$ was reduced by 1. We need to verify that neither side was reduced by more than 1, i.e.
we claim that
the remaining non-enumerated edges satisfy the relation
\begin{equation}\label{sum condition for v_0}
\sum_{ \wt v: \ol v_1 \in \sigma_{n-1}(\wt v)      }    \sum_{w\in W_{\wt v}(n)}\wt f_{u,w,\ol v}^{(n)}    -1 = \sum_{ w'\in W_{\ol v_1 }'(n)}  \ol f_{u,w',}^{(n)} -1 \, .
\end{equation}

The choice of $w'\in [\ol v_1, \wt v_1,n]$ actually
 means that we take the edge from $[\ol v_0, \wt v_0,n]$ to $[\ol  v_1, \wt v_1,n]$ in the associated graph $\mathcal H_n$.
Note that
 $\ol  v_1 \not\in  \wt \sigma_{n-1}(\wt v_1)$, otherwise there would be a
loop at $[\ol v_1, \wt v_1,n]$ in $\mathcal H_n$, a contradiction to our assumption.
 This is why there is exactly one edge removed in  each side
 of  (\ref{sum condition for v_0})
 so that our resulting row
 still satisfies  (\ref{general equal cardinalities one a}).
This completes the first step of the construction.

We are now at the vertex $[\ol v_1, \wt v_1,n]$ in $\mathcal H_n$. To repeat the above procedure,
note that we now have a  `new', reduced $u$-th row of $F_n$ - namely, the entry $f_{\ol v_0}{(n)}$ has been reduced by one. Thus the crossing numbers of the vertices of $\mathcal H_n$ change (one crossing number is reduced by one). Also note that in this new reduced row, $\ol f_{u,w'}^{(n)}= f_{u,w'}^{(n)} -1$; in other words, with each step of this algorithm the row we are working with changes, and the vertex $w$ such that $\ol f_{u,w}^{(n)}= f_{u,w}^{(n)} -1$ changes.  For, the vertex such that $\ol f_{u,w}^{(n)}= f_{u,w}^{(n)} -1$  belongs to the vertex in $\mathcal H_n$ where we currently are, and this changes at every step of the algorithm.
Let us assume that all crossing numbers are still positive for the time being to describe the second step of the algorithm.

We apply the above described procedure again, this time to $w'=s(e_1)$, to
show how we should proceed to complete the next step.
Consider the set
\[\{ \wt f_{u,w}^{(n)} : w \in [\ol v, \wt v, n]: \ol v\in \sigma_{n-1}(\wt v_1)  \}\, ,\]
and let $\wt f_{u,w''}^{(n)}$ be the maximum of this set,
 where $w''\in [\ol v_3,  \wt v_3 , n]$. Once again,
if there are several
 entries that are the maximal value, we chose one arbitrarily
 amongst them. Take any edge $e_2 \in E(w'', u)$. In the case where  $\wt e_u \in
 E(w'', u)$, we choose $e_2 \neq \wt e_u$. Assign the number $2$ to $e_2$
 so that $e_2$ becomes the successor of $e_1$.

We note that in the collection of relations (\ref{general equal cardinalities two a}), indexed by
the vertices $u$, $\wt v_1$ and $w'=s(e_1)$, one entry was `removed' from each side of the relation:
on the right hand side, the entry $\wt f_{u,  w', \ol v_2}^{(n)}$ was reduced by 1.

In the collection of relations (\ref{general equal cardinalities one a})
we have worked with the relation
defined by $u$ and $\ol v_2$. On the left hand side,
the entry $\wt f_{u,w', \ol v_2}^{(n)}$ was reduced by 1, and on the right hand side, $\ol f_{u,w''}^{(n)}$ was reduced by 1.
 As we saw in  (\ref{sum condition for v_0}), the relevant relation in (\ref{general equal cardinalities one a}) becomes
\begin{equation}\label{sum condition for v_1}
\sum_{ \wt v: \ol v_2 \in \sigma_{n-1}(\wt v)      }    \sum_{w\in W_{\wt v}(n)}\wt f_{u,w,\ol v}^{(n)}    -1 = \sum_{ w'\in W_{\ol v_2 }'(n)}  \ol f_{u,w'}^{(n)} -1 \, .
\end{equation}

We remark also that the choice that we made of $w''$ (or $e_2$) allows
us to continue the existing path (in fact, the edge) in $\mathcal H_n$
from $[\ol v_0, \wt v_0, n]$ to $[\ol v_1, \wt v_1,n]$ with the edge from $[\ol v_1, \wt v_1,n]$ to
$[\ol v_2 , \wt v_2,n]$, where $\wt v_2$ is defined by the property that
$s(e_2) \in [\ol v_2, \wt v_2 , n]$.

This process can be continued. At each step we apply the following rules:

(1) the edge $e_i$, that must be chosen next after $e_{i-1}$, is taken
from the set $E(w^*, u)$ where $w^*$ is such that $f_{u,w^*}^{(n)}$
is maximal
amongst $f_{u,w}^{(n)}$, as $w$ runs over $[\ol v, \wt v,n]$ where $\ol v\in \sigma_n (\wt v_{i-1})$, and

(2) the edge $e_i$ is always taken not equal to $\wt e_u$ unless no more
 edges except $\wt e_u$ are left.

After every step of the construction, we see that the following statements hold.

(i)  Relations (\ref{general equal cardinalities two a}), (with $\wt v= \wt v_i$)  and (\ref{general equal cardinalities one a}) (with $\ol v = \ol v_i$) remain true when we treat them as the number of non-enumerated edges left in $r^{-1}(u)$. In other words, when a pair of vertices $\wt v_i$ and $\ol v_i$ is considered, we reduce by 1 each side of the relevant relations.

(ii) The used procedure allows us to build a path $p$ from the starting
vertex $[\ol v_0, \wt v_0,n]$ going through other vertices of the graph
$\mathcal H_n$ according to the choice we make at each step. We need to guarantee that at each step, we are able to move to a vertex in $\mathcal H_n$ whose crossing number is still positive (unless we are at the terminal stage). As long as the crossing numbers of vertices in  $\mathcal H_n$ are positive, there is no concern. Suppose thought that we land at a (non-terminal) vertex $[\ol v, \wt v, n]$ in $\mathcal H_n$  whose crossing number is one (and this is the first time this happens). When we leave this vertex, to go to $[\ol v', \wt v' ,n]$,  the crossing number for $[\ol v, \wt v, n]$ will become 0 and therefore it will no longer be a vertex of $\mathcal H_n$. {\em Thus at this point, with each step, the graph $\mathcal H_n$ is also changing (being reduced).} We need to ensure that there is a way to continue the path out of $[\ol v', \wt v' ,n]$.
 Since
$$
\sum_{w\in W_{\wt v'}(n) } \wt f_{u,w}^{(n)}\geq      P_u[\ol v', \wt v' ,n]\geq 1,
$$
then  for some $\ol v \in \sigma_{n-1}(\wt v')$,
 $$\sum_{ \wt v: \ol v \in \sigma_{n-1}(\wt v) }    \sum_{w\in W_{\wt v}(n)}\wt f_{u,w,\ol v}^{(n)} \geq 1,
$$
 so that by the balance relations, $ \sum_{ w'\in W_{\ol v }'(n)}  \ol f_{u,w'}^{(n)} \geq 1 \, $.
 If the crossing number of all the vertices $[\ol v, *,n]$ have been reduced to 0, then this means that $\sum_{w' \in W_{\ol v}(n) } \ol f_{u,w'}^{(n)}= 1$,
 this tells us that we have to move into the terminal vertex for the last time.
Then the balance relations, which continue to be respected, ensure we are done. Otherwise,
 the balance relations guarantee that $\sum_{w' \in W_{\ol  v}(n) } \ol f_{u,w'}^{(n)}> 1$, which means there is a valid continuation of our path out of $[\ol v', \wt v' ,n]$ and to a new vertex in $\mathcal H_n$, and we are not at the end of the path. It is these balance relations which always ensure that the path can be continued until it reaches its terminal vertex.

(iii) In accordance with (i), the $u$-th row of $F_n$ is transformed by a sequence
of steps in such a way that entries of the obtained rows form decreasing
 sequences. These entries show the number of non-enumerated edges
 remaining  after the completed steps. It is clear that, by the rule used above, we decrease the largest entries first. It follows from the simplicity of the diagram that, for sufficiently many steps, the set $\{s(e_i)\}$ will contain all vertices $v_1,..., v_d$ from $V_n$. This means that the transformed $u$-th row consists of entries which are strictly less than those of $F_n$. After a number of steps the $u$-th row will have a form where the difference between any two entries is $\pm 1$. After that, this property will remain true.

(iv) It follows from (iii) that we finally obtain that all entries of
 the resulting  $u$-th row are zeros or ones. We apply the same procedure to enumerate the remaining edges from $r^{-1}(u)$ such that the number $|r^{-1}(u)| - 1$ is assigned to the edge $\wt e_u$. This means that we  have constructed the word $W_u = s(\ol e_u)s(e_1) \cdots s(\wt e_u)$, ie we have ordered $r^{-1}(u)$.

Looking at the path $p$ that is simultaneously built in $\mathcal H_n$, we
 see that the number of times this path comes into and leaves  a vertex $[\ol v, \wt v,n]$
is precisely that vertex's crossing number $P_u ([\ol v, \wt v,n])$. The path  $p$ is an Eulerian path of $\mathcal H_n$ that finally arrives to the vertex of $\mathcal H_n$ defined by $s(\wt e_u)$.

Case II: there is a loop in $\mathcal H_n$. To deal with this case, we
have to refine the described procedure to avoid a possible situation
when the algorithm cannot be finished properly. Suppose that the graph
$\mathcal H_n$ has some loops.

We start as in Case I, and continue until we have arrived to a vertex $[\ol v_1, \wt v_1,n]$, where, for the first time,
$[\ol v_1, \wt v_1,n]$ has a successor  $[\ol v_2, \wt v_2,n]$ with a loop, ie $\ol v_2 \in \sigma_n (\wt v_2)$.
If $[\ol v_2, \wt v_2,n]$
 has  crossing number zero, - ie it is the terminal vertex - and we are not at the terminal stage of defining the order, we ignore this vertex and continue as in Case I.
If $[\ol v_2, \wt v_2,n]$
 has a positive crossing number, i.e.  $P_u (  [\ol v_2, \wt v_2,n]   )>0  $,  then at this point, we continue the path to
  $[\ol v_2, \wt v_2,n]$ , and then
traverse this loop
$(\sum_{w \in [\ol v_2, \wt v_2,n]} \wt f_{u,w,\ol v_2}^{(n)} )-1$ times.
 This means we are traversing this loop enough times that it is effectively no longer part of the resulting $\mathcal H_n$ that we have at the end of this step - we will no longer need, or even be able, to traverse the loop.

 Looking at the relation
\begin{equation}\label{loop equation}
\sum_{ \wt v: \ol v_2 \in \sigma_{n-1}(\wt v)      }    \sum_{w\in W_{\wt v}(n)}\wt f_{u,w,\ol v_2}^{(n)}= \sum_{ w'\in W_{\ol v_2 }'(n)}  \ol f_{u,w'}^{(n)}, \,
\end{equation}
we see that by the time we have arrived at the vertex $[\ol v_2, \wt v_2,n] $, traversed it exactly  $(\sum_{w \in [\ol v_2, \wt v_2,n]} \wt f_{u,w,\ol v_2}^{(n)} )-1$   times, and left it,
we see that we
have removed  $\sum_{w \in [\ol v_2, \wt v_2,n]} \wt f_{u,w,\ol v_2}^{(n)} $  from each side of (\ref{loop equation}).
  We consequently
 enumerate  all edges whose source lies in $[\ol v_2, \wt v_2, n] $ in {\em any} arbitrary  order.

  We also need to ensure that once we have  `removed' the loop at   $[\ol v_2, \wt v_2,n] $   from   the graph $\mathcal H_n$, we  do not disrupt future movement of our path, i.e. we do not disconnect $\mathcal H_n$ in a damaging way. To see this, suppose we have a loop at $[\ol v_2, \wt v_2,n]$ whose crossing number is positive. If $[\ol v_1, \wt v_1, n]$ is a (non-looped) vertex with a positive crossing number which has $[\ol v_2, \wt v_2,n]$ as a successor, then for some
$[\ol v_3, \wt v_3,n]\neq [\ol v_2,\wt v_2,n]$ with
$\ol v_3\in \sigma_{n-1} (\wt v_1)$, the vertex $[\ol v_3, \wt v_3,n]$ will (if we are not at the terminal stage) have a positive crossing number. This is because of our discussion above concerning (\ref{loop equation}): the crossing number at the looped vertex appears  on both sides, and cancels. So if $[\ol v_1, \wt v,n]$ has a positive crossing number, this contributes positive values to the left hand side of (\ref{loop equation}); and so there  is some vertex $[\ol v_3, \wt v_3,n]$ with a positive value on the right hand side.  All this means that we are able to continue our path out of the looped vertex $[\sigma_n(\wt v), \wt v,n] $ when we arrive there at some future stage.

 We now revert to the old procedure. We are at the vertex
  $[\ol v_2, \wt v_2,n] $  in $\mathcal H_n$. If none of its successors have a loop, we revert to the algorithm in Case I, and continue with that algorithm until we reach a vertex in $\mathcal H_n$ one of whose successors has a loop, and then repeat the procedure described in Case II. We continue until we have defined the order on $r^{-1} (u)$.
To summarize the general procedure, we notice that, constructing the Eulerian
 path $p$, the following rule is used: as soon as $p$ arrives before a loop
 around a vertex  in $\mathcal H_n$, then $p$ makes as many loops
 around that vertex as needed so that this loop never needs or can be used again.Then $p$
 leaves the looped vertex and proceeds to a vertex  according to the procedure  in
 Case I, or, if there is as a  follower a vertex with a loop, according to the procedure in Case II.

As noticed above, the fact that all edges $e$ from $r^{-1}(u)$ are enumerated is equivalent to defining a word formed  by the sources of $e$. In our construction, we obtain the word $w(u, n, n+1)= s(\ol e_u) s(e_1)\cdots s(e_j)\cdots \cdots s(\wt e_u)$.

Applying these arguments to every vertex $u$ of the diagram, we define
 an ordering $\om$ on $B$. That $\om$ is perfect follows from Lemma
 \ref{perfect_order_characterisation}: we chose $\om$ to have skeleton $\mathcal F$,
 and for each $n$, constructed all words $w(v,n,n+1)$  to correspond to
 paths in $\mathcal H_n$. The result follows.
\end{proof}

\begin{example}\label{basic_infinite_rank_example_order} We continue with Examples \ref{basic_infinite_rank_example} and \ref{basic_infinite_rank_correspondence}, defining an order on $r^{-1}(v_2)$ where $v_2 \in V_3$, if $(f_{v_2,v_1}^{(3)},f_{v_2,v_2}^{(3)},f_{v_2,v_3}^{(3)})=(1,2,1)$. In what follows we drop the superscript $(3)$. This simple example illustrates why loop in the graphs $\mathcal H_n$ can cause a problem. The graph $\mathcal H = \mathcal H_3$ is shown in Figure \ref{H_3}; recall that $v_1\in [v_1,v_1]$, $v_2 \in [v_2,v_2]$ and $v_3 \in [v_1,v_2]$. Since all the maps $\sigma_n$ are point maps, the
system of relations (\ref{general equal cardinalities two a}) becomes trivial.
The balance relations (\ref{general equal cardinalities one a}) become
\[\wt f_{v_2,v_1}= \ol f_{v_2,v_2},  \mbox{ and }
\wt f_{v_2,v_2} + \wt f_{v_2,v_3}= \ol f_{v_2,v_1}+ \ol f_{v_2,v_3}  \]
respectively, and our vector $(1,2,1)$ satisfies these constants. The only valid choice of an ordering of  $r^{-1}(v_2)$ obtained using our algorithm is  $w(v_2,2,3)= v_2v_3v_1v_2$, and in fact this is the only valid ordering possible.
\end{example}

\begin{figure}
\centerline{\includegraphics[scale=0.9]{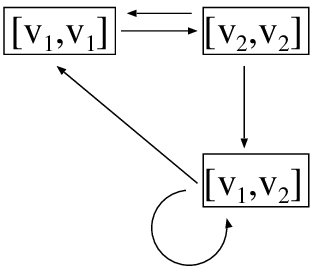}}
\caption{The graph $H_3$ for Example \ref{basic_infinite_rank_example_order}}\label{H_3}
 \end{figure}

\begin{example} \label{basic_set_map_correspondence_example_order} We continue with Example \ref{basic_set_map_correspondence_example}, and illustrate how to define an order on $r^{-1}(v_1)$ where $v_1 \in V_4$, if $(f_{v_1,v_1}^{(4)},f_{v_1,v_2}^{(4)},f_{v_1,v_3}^{(4)},f_{v_1,v_4}^{(4)})=(4,2,2,3)$. In what follows we drop the superscript $(4)$. The graph $\mathcal H = \mathcal H_4$ is shown in Figure \ref{H_4}; recall that $v_1\in [v_1,v_1]$, $v_2 \in [v_2,v_1]$, $v_3 \in [v_2,v_2]$ and $v_4 \in [v_3,v_3]$.   The (nontrivial part of the) system of relations
(\ref{general equal cardinalities two a})
become
\[  \wt f_{v_1,v_1}=  \wt f_{v_1,v_1,v_2}+ \wt f_{v_1,v_1,v_3}  \mbox{ and }\wt f_{v_1,v_2}=  \wt f_{v_1,v_2,v_2}+ \wt f_{v_1,v_2,v_3}, \]and the balance relations (\ref{general equal cardinalities one a}) become, with $\ol v=v_1, \,\,  v_2$ and $v_3$
\[\wt f_{v_1,v_4}= \ol f_{v_1,v_1}, \,\,  \wt f_{v_1,v_1,v_2} +  \wt f_{v_1,v_2,v_2} = \ol f_{v_1,v_3} + \ol f_{v_1,v_2} \mbox{ and }
\wt f_{v_1,v_1,v_3} + \wt f_{v_1,v_3}+ \wt f_{v_1,v_2,v_3} = \ol f_{v_1,v_4}  \]
respectively.
If we let
\[  \wt f_{v_1,v_1}=  \wt f_{v_1,v_1,v_2}+ \wt f_{v_1,v_1,v_3}  = 2+1 \mbox{ and }\wt f_{v_1,v_2}=  \wt f_{v_1,v_2,v_2}+ \wt f_{v_1,v_2,v_3} = 2+0, \]
then the balance relations are satisfied. A valid choice of an ordering on $r^{-1}(v_1)$ obtained using our algorithm is  $w(v_1,3,4)= v_1v_2^2v_3v_4v_1 v_3v_4 v_1v_4 v_1$; the other is $w(v_1,3,4)= v_1v_2^2v_3v_4v_1 v_4 v_1 v_3 v_4 v_1$. Note that there are other valid choices of orderings on $r^{-1}(v_1)$, but they are not achieved with this algorithm.
\end{example}

\begin{figure}
\centerline{\includegraphics[scale=0.9]{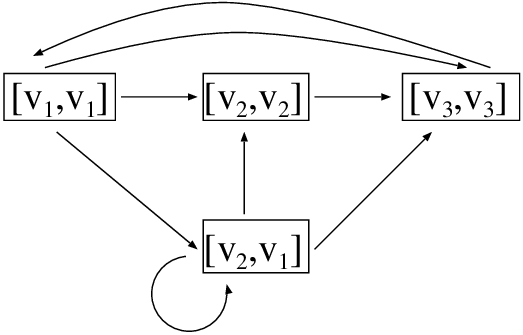}}
\caption{The graph $H_4$ for Example \ref{basic_set_map_correspondence_example_order}}\label{H_4}
 \end{figure}

As a corollary we identify the conditions needed so that a perfect order in $\mathcal P^*_B$ is supported by $B$.
 Note  that one can talk of a skeleton and correspondence $\mathcal F$ and $\sigma$  of being capable of generating orders that belong to $\mathcal P_B^*$: either all perfect orders $\om$ having skeleton $(\mathcal F, \sigma)$ belong to $\mathcal P_B^*$, or none of them do.

\begin{corollary}\label{analogous_special}
Let $B$ be a Bratteli diagram with incidence matrices $(F_n)$.  Let $\mathcal F$ be a skeleton on $B$ and $\sigma$ an associated correspondence that can  generate orders in $\mathcal P_B^*$, and
suppose that all associated graphs $\mathcal H_n$ are positively  strongly connected. Suppose also that for each $M$ (and hence each sequence $(\wt v_n = v_n(M))$), there exists an $n_0$ such that for any $n \geq n_0$,  any $m >n$, and any $u \in V_m$,
 the entries of incidence matrices $(F_n)$ satisfy condition (\ref{sums of entries}).
Then there is a perfect ordering $\om$ on $B$ such that $\mathcal F = \mathcal F_\om$ and the Vershik map $\varphi_\om$ satisfies the relation $\varphi_\om = \sigma$ on $X_{\max} (\mathcal F)$.
\end{corollary}

In the remaining part of this section, we will consider the class of Bratteli diagrams $\mathcal A$ that is close, by its structure, to diagrams of finite rank. We refer to the notation in Definition \ref{aperiodic_definition}.

Take a Bratteli diagram $B$ and let the columns of $A_{n}^{(i)}$ be indexed by vertices $V_n^{(i)}$. Define
$X_B^{(i)} = \{x =(x_n) \in X_B: s(x_n)\in V_n^{(i)} \mbox{ for each } n  \}$, and let $X_B^{(a)}= X_B \backslash \bigcup_{i}X_B^{(i)}$. We will call $X_B^{(i)}$ the $i$-th minimal component of $X_B$ (here minimality is considered with respect to the tail equivalence relation $\mathcal E$). If $B \in \mathcal A$, and has incidence matrices of the form as in (\ref{aperiodic_matrix_form}) we will say that $B$ has {\em $k$ minimal components}.

Given a skeleton $\mathcal F$ on $B$, let $X_{\max}^{(i)} (\mathcal F)= X_{\max} (\mathcal F) \cap X_B^{(i)}$;
 define $X_{\min}^{(i)} (\mathcal F)$ similarly. Also, if $\om$ is an order on $B$, define $X_{\max}^{(i)} (\om)$ and $X_{\min}^{(i)} (\om)$ analogously. Note that for any statement that we make about a skeleton, we can  make an analogous statement about an order - simply consider the skeleton associated with the well-telescoped order.

The following lemma is straightforward.

\begin{lemma}\label{minimal_components}
Let $B\in \mathcal A$.
Then
\begin{enumerate}
\item
If $\mathcal F$ is a skeleton on $B$, then, for each $i$, the sets
$X_{\max}^{(i)} (\mathcal F)$ and $X_{\min}^{(i)} (\mathcal F) $ are closed.
\item
If $\om$ is a perfect order on $B$, then, for each $i$, $\varphi_\om: X_{\max}^{(i)}(\om) \rightarrow X_{\min}^{(i)}(\om)$ is a homeomorphism.
\end{enumerate}
\end{lemma}

We use Proposition \ref{connectedness and goodness} to prove the following generalization of Proposition 3.26 in \cite{bky}.

\begin{proposition}\label{corollary_1}
Let $B\in \mathcal A$ have $k$ minimal components. Suppose that
 $C_{n}$ is a $d\times d$ matrix where $1\leq d\leq k-1$.
If $k=2$, then there are perfect orderings on $B$ only if $C_{n}=(1)$ for
all but finitely many $n$. If $k>2$, then there is no perfect ordering on $B$.

\end{proposition}
\begin{proof}
We first claim that in $\mathcal H_n$, there are
$k$ connected components of vertices $T_{n}^{(1)},
\ldots T_{n}^{(k)}$, such that there are no edges from vertices in $T_{n}^{(i)}$
to vertices in $T_{n}^{(j)}$ if $i\neq j$. To see this, if $1\leq i \leq k$, let
$T_n^{(i)} = \{[\ol v, \wt v,n]: \ol v \in V_n^{(i)}, \wt v\in
V_n^{(i)}\}$.
By Lemma \ref{minimal_components}, for large $n$,  if $\wt v \in \wt V_n^{(i)}$ it is not possible that $\ol v\in \sigma_n(\wt v)$ if $\ol v \not\in \ol V_n^{(i)}$.

If $\om$ is a perfect order on $B$,  we assume that    $(B, \om)$ is well telescoped and  has skeleton $\mathcal F_\om$. (Otherwise we work with the diagram $B'$ on which $L(\om)$ is well telescoped: Note that if $B$ has
incidence matrices of the given form, then so does any
telescoping.)
By Proposition \ref{connectedness and goodness}, the graphs $\mathcal H_n$ are weakly connected. The only way that this can happen is if there are $k-1$ remaining vertices  in $\mathcal H_n$ (so that $d$ must equal $k-1$),
each have one incoming edge from one of the components  $T_{n}^{(i)}$, and one outgoing edge into one of the components $T_{n}^{(j)}$.   Each of these remaining vertices in $\mathcal H_n$ corresponds to exactly one vertex in $V_n\backslash \bigcup_{i=1}^{k}V_n^{(i)}$.
Thus at least one of the components, say
$T_{n}^{(1)}$, has no incoming edges. Take now a vertex $t$, not belonging to any of the $T_n^{(j)}$'s, such that if $v\in t$, then $s(\ol e_v)  \in T_n^{(p)}$, where $p\neq 1$.  Since the row in $F_n$ corresponding to $v$ is strictly positive, we have a contradiction as there is no path from $T_n^{(p)}$ to $T_n^{(1)}$.
\end{proof}

\section{The infinitesimal subgroup of diagrams that support perfect orders}

We will use our results from Section \ref{characterization} to give an alternative proof of the  following result that was proved in \cite{gps} (Corollary 2). We recall that one can associate the so called {\em dimension group} with each simple Bratteli  diagram $B =(V,E)$. Let $(F_n)$ be the sequence of incidence matrices of $B$, then the dimension group $G$ is defined as the inductive limit:
$$
G := \varinjlim_{n\to \infty} \mathbb Z^{|V_n|} \stackrel{F_n}\longrightarrow \mathbb Z^{|V_{n+1}|}.
$$ For all information on dimension groups, we refer the reader to \cite{effros}.

If $G$ is a simple dimension group, the elements  $g$ of $G$  of the infinitesimal subgroup $Inf(G)$ are defined by the relation: $\tau(g) =0$ for every normalized  trace $\tau$ on $G$. Any normalized trace is defined by a probability measure invariant with respect to the tail equivalence relation.
If $B'$ is a telescoping of $B$, then their dimension groups $G$, $G'$ are  group order  isomorphic, and this isomorphism maps $Inf(G)$ onto $Inf(G')$.

\begin{theorem}\label{general-finite}
Let $B$ be a simple Bratteli diagram and $G$ its dimension group. Suppose $\om \in \mathcal P_{B}\cap \mathcal O_{B}(j)$, that is there is a perfect order $\om$ on $B$ with exactly $j$ maximal paths and $j$ minimal paths where $j\geq 2$.  Then  $Inf(G)$, the infinitesimal subgroup of $G$, contains a subgroup isomorphic to  $\mathbb Z^{j-1}$.
\end{theorem}

\begin{proof}
Given a simple Bratteli diagram $B$ and $\om \in \mathcal P_{B}\cap \mathcal O_{B}(j)$, we assume that $(B,\om)$ is well telescoped. For if not, we will telescope $(B,\om)$ to $(B',\om')$, and working with $(B',\om')$, we shall show that if $G'$ is the dimension group defined by $B'$, then $Inf(G')$ contains a subgroup isomorphic to  $\mathbb Z^{j-1}$. Since $G$, $G'$ are order isomorphic groups with the isomorphism mapping $Inf(G)$ to $Inf(G')$, then $Inf(G)$ will also contain a subgroup isomorphic to $\mathbb Z^{j-1}$.

(1) In the proof, we will use the notation defined in Sections \ref{skeletons} and \ref{characterization}.
This means that we will freely operate with such objects as the skeleton $\mathcal F_\om$, correspondence $\sigma = (\sigma_n)$,  sets of maximal and minimal vertices $\wt V_n, \ol V_n$, partitions $W(n), W'(n)$, sets $\wt E(W_{\wt v}(n), u), \ol E(W'_{\ol v}(n), u)$, maximal and minimal finite paths $\wt e_v, \ol e_v, \wt  e(V_n, u), \ol e(V_n, u)$ (defined just before Lemma \ref{successor}), etc.

Fix a level $n$ such that $n> n_0$ where $n_0$ is as in the first statement of Corollary \ref{necessary condition}, so that $|\wt V_i| = |\ol V_i| = j$ for all $i \geq n-1$, and take a maximal vertex $\wt v^*\in \wt V_{n-1}$. We construct a sequence of vectors $(\e_{\wt v^*}^{(n+k)})_{k\geq 1}$ with $ \e_{\wt v^*}^{(n+k)} \in \mathbb Z^{|V_{n+k}|}$ as follows. Take first a vertex $v \in V_{n+1}$ and set
\begin{equation}\label{balance3}
\varepsilon_{\wt v*}^{(n+1)}(v):= \left\{
\begin{array}{cc}
- 1
& \mbox{if }   s(\ol e_v) \in W'_{\sigma_{n-1}(\wt v^*)} (n), s(\wt e_v) \notin W_{\wt v^*} (n),
\\
1
& \mbox{if } s(\ol e_v) \notin W'_{\sigma_{n-1}(\wt v^*)} (n), s(\wt e_v) \in W_{\wt v^*} (n),
 \\
0 & \mbox{ otherwise }
  \end{array}
  \right.
\end{equation}
to obtain the $v$-th entry of $\e_{\wt v^*}^{(n+1)}$. In general, let $v$ be any vertex from $V_{n+k}$. Then we define $\e_{\wt v^*}^{(n+k)}$ as follows:
\begin{equation}\label{balance4new}
\varepsilon_{\wt v*}^{(n+k)}(v):= \left\{
\begin{array}{cc}
- 1
& \mbox{if }   s(\ol e(V_n, v)) \in W'_{\sigma_{n-1}(\wt v^*)} (n), s(\wt e(V_n, v)) \notin W_{\wt v^*} (n),
\\
1
& \mbox{if } s(\ol e(V_n, v)) \notin W'_{\sigma_{n-1}(\wt v^*)} (n), s(\wt e(V_n, v)) \in W_{\wt v^*} (n),
 \\
0 & \mbox{ otherwise. }
  \end{array}
  \right.
\end{equation}

(2) We will show that for any $k \geq 1$
\begin{equation}\label{F epsilon}
F_{n+k} \varepsilon_{\wt v*}^{(n+k)}= \e_{\wt v*}^{(n+k+1)}.
\end{equation}
To prove (\ref{F epsilon}), we use another representation of entries of the vector
$\varepsilon_{\wt v*}^{(n+k)}(v)$. Indeed, since $\om$ is perfect we have that  relation (\ref{sums of entries}) in Section \ref{characterization} holds.
Also, if $F(n, n+k) = F_{n+k-1}\circ \cdots \circ F_n, k\geq 1$, and
 $u\in V_{n+k}$,  then relation   (\ref{sums of entries}) becomes
\[
\sum_{w \in W_{\wt v}(n) }\wt f_{u,w}^{(n, n+k)} = \sum_{w'\in W'_{\sigma_{n-1}(\wt v)}(n)} \ol f _{u, w'}^{(n, n+k)}, \ \ \ \ u \in V_{n+k}\, .
\]
It is straightforward to check that
\begin{equation}\label{vector epsilion}
\varepsilon_{\wt v*}^{(n+k)}(v) = \sum_{w \in W_{\wt v^*}(n)}  f^{(n, n+k)}_{v,w} -
\sum_{w' \in W'_{\sigma_{n-1}(\wt v^*)}(n)}  f^{(n, n+k)}_{v,w'},\ \ v \in V_{n+k}.
\end{equation}
Then (\ref{F epsilon}) can be proved by induction. Indeed, compute for $k=1$
\begin{eqnarray*}
F_{n+1}\varepsilon_{\wt v*}^{(n+1)}(v) & = & \sum_{u\in V_{n+1}} f^{(n+1)}_{v, u}
\left( \sum_{w \in W_{\wt v^*}(n)}  f^{(n)}_{u,w} -
\sum_{w' \in W'_{\sigma_{n-1}(\wt v^*)}(n)}  f^{(n)}_{u,w'}\right) \\
   &=&  \sum_{w \in W_{\wt v^*}(n)}  f^{(n, n+2)}_{v,w} -
\sum_{w' \in W'_{\sigma_{n-1}(\wt v^*)}(n)}  f^{(n, n+2)}_{v,w'}\\
& = & \varepsilon_{\wt v*}^{(n+2)}(v),  \ \ v \in V_{n+2}.
\end{eqnarray*}
The induction step can be computed in a similar way. We omit the details.

(3) It follows from relation (\ref{F epsilon}) that every vector $\varepsilon_{\wt v^*}^{(n+1)}, \wt v^* \in \wt V_{n-1},$ generates the  element
$g_{\wt v^*}=(\varepsilon_{\wt v^*}^{(n+1)}, \varepsilon_{\wt v^*}^{(n+2)}, \varepsilon_{\wt v^*}^{(n+3)}, \ldots  )$ of the dimension group $G$ of $B$.

Next,  since $W(n)$ and $W'(n)$ constitute partitions of $V_n$, we have
\begin{eqnarray*}
\sum_{\wt v\in \wt V_{n-1}} \e^{(n+k)}_{\wt v}(v)  &=& \sum_{\wt v\in \wt V_{n-1}}\left( \sum_{w \in W_{\wt v^*}(n)}  f^{(n, n+k)}_{v,w} -
\sum_{w' \in W'_{\sigma_{n-1}(\wt v^*)}(n)}  f^{(n, n+k)}_{v,w'}\right) \\
& =& \sum_{w \in V_n}  f^{(n, n+k)}_{v,w} - \sum_{w' \in V_n}  f^{(n, n+k)}_{v,w'} \\
& = & 0
\end{eqnarray*}
for any $k$ and $v$. That is the vectors $\{\e^{(n+k)}_{\wt v} : \wt v \in \wt V_{n-1}\}$ are linearly dependent.

On the other hand, {\em we claim that any subset of this set containing $j-1$  vectors is linearly independent over $\Z$}. To simplify our notation we consider the set of vectors $\{\e^{(n+1)}_{\wt v} : \wt v \in \wt V_{n-1}\}$ only, since the case for the vectors $\{\e^{(n+k)}_{\wt v} : \wt v \in \wt V_{n-1}\}$ where   $k>1$ is considered similarly. Suppose that $\wt V_{n-1} = \{ \wt v_1, \ldots, \wt v_j  \}.$
 Consider the  $|V_{n+1}| \times j$-matrix  whose $i$-th column is the vector
$\e^{(n+1)}_{\wt v_i} $.  We fix a vertex $u\in V_{n+1}$ and look at the $u$-th row   $R_{u}$  of this matrix; this is  the row formed by the $u$-th entries of the vectors $\{\e^{(n+1)}_{\wt v_i}\}, 1 \leq i \leq j$.  The definition of  the entries of the row $R_{u}$  (see (\ref{balance3})) shows that either they  are all  0
or each value 1 and $-1$ is taken exactly once, and all remaining entries in this row are zero. Note that since this property holds for any vertex $u$, we have another proof of the fact  that the sum of all vectors $\{\e^{(n+1)}_{\wt v} \}$ is zero.
Also, for every $\wt v$,  the strong connectivity of the graph $\mathcal H_n$ implies that  we have  $|\{u \in V_{n+1} : \e_{u,\wt v}^{(n+1)} =1 \}| \geq 1$ and
 $|\{u \in V_{n+1} : \e_{u,\wt v}^{(n+1)}    = -1 \}| \geq 1$ (but these sets may be of different cardinalities).

We pick any vector $\e_{\wt v}^{(n+1)}$ in the set  $\{\e^{(n+1)}_{\wt v_i}\}$ and show that the set  of remaining vectors $\{\e^{(n+1)}_{\wt v_i}\} \setminus \e_{\wt v}^{(n+1)}$  is linearly independent. Indeed, assume that
\begin{equation}\label{linear independence}
\sum_{\wt v_i \neq \wt v} m_{\wt v_i} \e^{(n+1)}_{\wt v_i} = 0;
 \end{equation}
 we shall show that all $m_{v_i}$'s in
  (\ref{linear independence}) must be equal to $m$. If one assumes that $m \neq 0$ we obtain two contradictory equalities $\sum_{\wt v \neq \wt v_0} \e^{(n+1)}_{\wt v} =0$ and $\sum_{\wt v \neq \wt v_0} \e^{(n+1)}_{\wt v} = - \e^{(n+1)}_{\wt v_0}$.  Pick any two maximal vertices $\wt v_p$ and $\wt v_t$ that are not  equal to $\wt v$. Firstly, by the strong connectivity of $\mathcal H_n$, there exist vertices $[\ol v_p, \wt v_p]$ and $[\ol v_t, \wt v_t]$ in $\mathcal H_n$ that do not have loops. Secondly, also by the strong connectivity of $\mathcal H_n$, we can find a path from $[\ol v_p, \wt v_p]$ to $[\ol v_t, \wt v_t]$ , and we can also assume that  for any vertex along this path, there are no loops (otherwise these vertices can be removed from the path, and we still have a valid path). Finally, note that if there is an edge from $[\ol v', \wt v']$ to $[\ol v^*, \wt v^*]$ and there is no loop at
$[\ol v^*, \wt v^*]$, then $m_{\wt v^*}=m_{\wt v'}$. The result follows.

(4) It remains to show that the elements $g_{\wt v}, \wt v\in \wt V_{n-1},$ belong to   the infinitesimal subgroup $Inf(G)$. Since $G$ is a simple dimension group, it suffices to check that $\tau(g_{\wt v}) = 0$ for any trace $\tau$ on $G$ or, equivalently, for any invariant measure $\mu$.

Fix a probability $\varphi_\om$-invariant measure $\mu$ on the path space of the diagram $B$. Consider the vector $ p^{(n)}= (p_v^{(n)}) \in \mathbb R^{|V_n|}$ whose entries are $\mu$-measures of a finite path (cylinder set) with source $v_0$ and range  $v$. This means that, in particular, $\mu(\wt e(v_0,w))  =p_w^{(n)}, w \in W_{\wt v}(n)$ and $\mu(\ol e (v_0, w')) =p_{w'}^{(n)}, w' \in W_{\sigma_{n-1}(\wt v)}(n)$. As noticed in \cite{bkms}, we have $F_n^T p^{(n+1)} = p^{(n)}, \forall n\in \N,$
or
$$
\sum_{v\in V_{n+1}} f_{v,w}^{(n)} p_v^{(n+1)} = p^{(n)}_w, \  \ w\in V_n.
$$

Compute $<p^{(n+1)}, \e ^{(n+1)}_{\wt v}>$ where $<\cdot, \cdot>$ is the inner product:
\begin{eqnarray*}
<p^{(n+1)}, \e ^{(n+1)}_{\wt v}> &= &<F_{n+1}^Tp^{(n+2)}, \e ^{(n+1)}_{\wt v}>\\
& = & <p^{(n+2)}, F_{n+1}\e ^{(n+1)}_{\wt v}>\\
&=& <p^{(n+2)}, \e ^{(n+2)}_{\wt v}> \\
& \cdots & \\
& = & <p^{(n+j)}, \e ^{(n+j)}_{\wt v}>
\end{eqnarray*}
for any $j$. On the other hand,
\begin{eqnarray*}
  <p^{(n+1)}, \e ^{(n+1)}_{\wt v}> &=& \sum_{v\in V_{n+1}}  p_v^{(n+1)} \left(\sum_{w\in W_{\wt v}(n)}f_{v,w}^{(n)} - \sum_{w'\in W'_{\sigma_{n-1}(\wt v)}(n)} f_{v,w'}^{(n)}\right) \\
  &= &\sum_{w\in W_{\wt v}(n)}\sum_{v\in V_{n+1}} f_{v,w}^{(n)} p_v^{(n+1)} -  \sum_{w'\in W'_{\sigma_{n-1}(\wt v)}(n)}\sum_{v\in V_{n+1}} f_{v,w'}^{(n)} p_v^{(n+1)}\\
   &=& \sum_{w\in W_{\wt v}(n)} p_w^{(n)} -  \sum_{w'\in W'_{\sigma_{n-1}(\wt v)}(n)} p_{w'}^{(n)}\\
&=& 0
\end{eqnarray*}
because
\begin{eqnarray*}
\sum_{w\in W_{\wt v}(n)} p_w^{(n)} &=& \mu(\bigcup_{w\in W_{\wt v}(n)}  \wt e (v_0, w))\\
& = & \mu(\varphi_\om(\bigcup_{w\in W_{\wt v}(n)}  \wt e (v_0, w)))\\
& = & \mu(\bigcup_{w'\in W'_{\sigma_{n-1}(\wt v)}(n)}  \ol e (v_0, w'))\\
& = & \sum_{w'\in W'_{\sigma_{n-1}(\wt v)}(n)} p_{w'}^{(n)}.
\end{eqnarray*}
This proves that $g_{\wt v}\in Inf(G)$.
\end{proof}

\begin{remark}
In the simpler case when $B$ is of finite rank,  then for each $n$, $\varepsilon_{\wt v*}^{(n)}= \varepsilon_{\wt v*}$, each of the latter $j$ vectors correspond to an infinitesimal,  and
$\varepsilon_{\wt v*}=-\sum_{\wt v \neq \wt v*} \varepsilon_{\wt v}$, while  $\{\varepsilon_{\wt v}: \wt v \neq \wt v*\}$ is a linearly independent set, so that there are $j-1$ identified copies of $\mathbb Z$ in the infinitesimal subgroup of $dim(B)$.
\end{remark}

\begin{example}
Let
\begin{equation}
F_n = \left(
      \begin{array}{cccc}
        f_{aa}^{(n)} & f_{ab}^{(n)}& \alpha^{(n)}& \alpha^{(n)}\\
        f_{ba} ^{(n)}& f_{bb}^{(n)} &  \beta^{(n)} & \beta^{(n)} \\
        f_{ca} ^{(n)}& f_{cb}^{(n)}& \gamma^{(n)}+1& \gamma^{(n)}\\
        f_{da}^{(n)}& f_{db}^{(n)}& \delta^{(n)}&  \delta^{(n)}+1\\
      \end{array}
    \right);
\end{equation}
then there exist orders $\om$ on $B$ that belong to
$ \mathcal P_{B}\cap \mathcal O_{B}(2)$, and such that the associated graph and correspondence is as in Figure \ref{first graph}, where $a \in [a,a],$ $ b \in  [b,b] $, $c \in [a,b]$ and $d \in [b,a]$. In this case,

\begin{equation}
\varepsilon_a= \left(
      \begin{array}{c}
        0\\
        0\\
        -1\\
        1\\
      \end{array}
    \right) \mbox{ and }
   \varepsilon_b= \left(
      \begin{array}{cccc}
        0\\
        0\\
        1\\
        -1\\
      \end{array}
    \right);
\end{equation}
so that $F_n \varepsilon_a = \varepsilon_a$ as claimed and  $\varepsilon_a$ corresponds to an element of $Inf(G)$.
\end{example}

\begin{example}
One can vary the given skeleton and correspondence in Theorem \ref{general-finite} to maximize the number of copies of $\mathbb Z$ that one can find in $Inf(G)$.
For example,
 if a finite rank simple diagram $B$  supports orders in  $\om \in \mathcal P_{B}(2)$ that can have either of the two possible graphs described in Example \ref{graph for two extremal paths},  this means that there is a copy of $\mathbb Z\times \mathbb Z$ in $Inf(G)$.  The incidence matrices of this diagram must have a very restrictive structure. For example, if $B$ has rank 4, then the  incidence matrices must be of the form

\begin{equation}
F_n = \left(
      \begin{array}{cccc}
        a_{n}^{(1)}+1 & a_{n}^{(1)} & a_{n}^{(2)}& a_{n}^{(2)}\\
        b_{n}^{(1)} & b_{n}^{(1)}+1  &  b_{n}^{(2)} & b_{n}^{(2)} \\
        c_{n}^{(1)} & c_{n}^{(1)} & c_{n}^{(2)}+1 & c_{n}^{(2)} \\
        d_{n}^{(1)}& d_{n}^{(1)} & d_{n}^{(2)} &  d_{n}^{(2)} +1 \\
      \end{array}
    \right)
\end{equation}
 Bratteli diagrams with these incidence matrices have 2 orders in $\mathcal P_{B}\cap \mathcal O_{B}(2)$, each with a different  associated  graph $\mathcal H$. This implies that they have (at least)  two independent infinitesimals, corresponding to the elements.

 \begin{equation}
\varepsilon= \left(
      \begin{array}{c}
        0\\
        0\\
        -1\\
        1\\
      \end{array}
    \right) \mbox{ and }
   \varepsilon'= \left(
      \begin{array}{cccc}
        1\\
        -1\\
        0\\
        0\\
      \end{array}
    \right).
\end{equation}

\end{example}

Next we extend Theorem \ref{general-finite} to diagrams supporting perfect orders  that belong to $\mathcal P_B^*$. This result overlaps Corollary 3 in \cite{gps}.

\begin{theorem}\label{general_P*}
Let $B$ be a simple Bratteli diagram and $G$ its dimension group. Suppose $\om$ belongs to  $ {\mathcal P_B^*} \backslash \bigcup_{j=1}^{\infty}\mathcal O_B(j)$. Then  $Inf(G)$, the infinitesimal subgroup of $G$, contains as a subgroup the free abelian group on countably many generators.
\end{theorem}

\begin{proof}
The proof is similar to that  of Theorem \ref{general-finite}.
  The second statement of   Corollary \ref{necessary condition} tells us that for each maximal path $M$, with $\wt v_n = v_n(M)$, there exists some level $n_0$ such that if $n\geq n_0 -1$, $\sigma_n (\wt v_n)$ is a singleton.
  We construct a sequence of vectors $(\e_{M}^{(n_0+k)})_{k\geq 1}$ with $ \e_{M}^{(n+k)} \in \mathbb Z^{|V_{n+k}|}$ as follows. Take first a vertex $v \in V_{n_0+1}$ and set
\[
\varepsilon_{M}^{(n_0+1)}(v):= \left\{
\begin{array}{cc}
- 1
& \mbox{if }   s(\ol e_v) \in W'_{\sigma_{n_0-1}(\wt v_{n_0-1})} (n_0), s(\wt e_v) \notin W_{\wt v_{n_0-1}} (n_0),
\\
1
& \mbox{if } s(\ol e_v) \notin W'_{\sigma_{n_0-1}(\wt v_{n_0-1})} (n_0), s(\wt e_v) \in W_{\wt v_{n_0-1}} (n_0),
 \\
0 & \mbox{ otherwise }
  \end{array}
  \right.
\]
to obtain the $v$-th entry of $\e_{M}^{(n_0+1)}$. In general, let $v$ be any vertex from $V_{n_0+k}$. Then we define $\e_{M}^{(n_0+k)}$ as follows:
\[
\varepsilon_{\wt v*}^{(n_0+k)}(v):= \left\{
\begin{array}{cc}
- 1
& \mbox{if }   s(\ol e(V_{n_0}, v)) \in W'_{\sigma_{n_0-1}(\wt v_{n_0-1})} (n_0), s(\wt e(V_{n_0}, v)) \notin W_{\wt v_{n_0-1}} (n_0),
\\
1
& \mbox{if } s(\ol e(V_{n_0}, v)) \notin W'_{\sigma_{n_0-1}(\wt v_{n_0-1})} (n_0), s(\wt e(V_{n_0}, v)) \in W_{\wt v_{n_0-1}} (n_0),
 \\
0 & \mbox{ otherwise. }
  \end{array}
  \right.
\]

As in (2) of Theorem \ref{general-finite}, we can show that
for any $k \geq 1$
\begin{equation}\label{FPepsilon}
F_{n_0+k} \varepsilon_{M}^{(n_0+k)}= \e_{M}^{(n_0+k+1)}.
\end{equation}
This means that  we can define, from relation (\ref{FPepsilon}), the  element
$g_{M}=(\varepsilon_{M}^{(n_0+1)}, \varepsilon_{M}^{(n_0+2)}, \varepsilon_{M}^{(n_0+3)}, \ldots  )$ of the dimension group $G$ of $B$. In this way,  we get a countably infinite collection of elements $\{g_M: M \mbox{ maximal}\}$.

The argument that the collection $\{g_M: M \mbox{ maximal}\}$ generates a free abelian group, as the case of (3) of  Theorem \ref{general-finite}, depends on the strong connectivity of the graphs $\mathcal H_n$. Take a finite set of maximal paths $(M_1,..., M_k)$ and suppose that there is a linear relation
\[
\sum_{i=1}^{k} m_{i} g_{M_i} = 0,
 \]
 where the $m_i$,s are nonzero. Let $\wt v_n^{i}= \wt v_n(M_i)$ and $\ol v_n^{i}= \ol v_n(M_i)$,  and choose an $N$ large enough so that $\sigma_{n}( \wt v_n^{i})$ is a singleton for each  $n\geq N-1$ and $1\leq i \leq k$. Consider the  $|V_{N+1}| \times k$-matrix  whose $i$-th column is the vector
$\e^{(N+1)}_{M_i} $.  We fix a vertex $u\in V_{N+1}$ and look at the $u$-th row   $R_{u}$  of this matrix; this is  the row formed by the $u$-th entries of the vectors $\{\e^{(n+1)}_{M_i}\}, 1 \leq i \leq k$.  The definition of  the entries of the row $R_{u}$   shows that  apart from at most one occurrence  of 1 and of $-1$,    they are all 0. Note that unlike the case in the proof of (3) of Theorem \ref{general-finite}, it is possible that only one of the values $1$, $-1$ appear in any row $R_u$.
Also, for every $\wt v$,  the strong connectivity of the graph $\mathcal H_n$ implies that  we have  $|\{u \in V_{n+1} : \e_{u,\wt v}^{(n+1)} =1 \}| \geq 1$ and
 $|\{u \in V_{n+1} : \e_{u,\wt v}^{(n+1)}    = -1 \}| \geq 1$.
The linear relation implies that $\sum_{i=1}^{k} m_{i} \e^{(N+1)}_{M_i} = 0. $ It follows that if a 1 occurs in the row $R_{u}$   a $-1$ must also occur; i.e. if $1$ occurs in $R_u$, then $u \in [\ol v_N^{i}, \wt v_N^{j}] $ for some $1\leq i,\, j \leq k$. Otherwise - if $R_u$ consists only of zeros - $u \in [\ol v, \wt v] $ with $\ol v \not \in \{ \ol v_N^1, \ldots ,\ol v_N^k\}$ and $\wt v \not \in \{ \wt v_N^1, \ldots ,\wt v_N^k\}$.  Thus, we have partitioned $\mathcal H_{N+1}$ into two disconnected sets of vertices, contradicting its strong connectivity.

The proof that each $g_M$ is an infinitesimal is now very similar to  part (4) of the proof of Theorem \ref{general-finite}.

\end{proof}

\proof[Acknowledgements]
The authors would like to thank Christian Skau for some helpful comments. Part of this work was completed while the authors were visiting  The University of Iowa (S.B.)   and The Universit\'{e} de Picardie Jules Verne (R.Y.). We are thankful to  these universities for their hospitality and support.

{\footnotesize
\bibliographystyle{alpha}
\bibliography{bibliography}
}

\end{document}